\documentclass{amsart}

\usepackage{amsmath,amssymb,amsthm}

\newtheorem{prop}{Proposition}[section]
\newtheorem{thm}[prop]{Theorem}
\newtheorem{lem}[prop]{Lemma}
\newtheorem{cor}[prop]{Corollary}

\theoremstyle{definition}

\newtheorem*{defn}{Definition}
\newtheorem*{ex}{Example}
\newtheorem*{rem}{Remark}
\newtheorem*{rems}{Remarks}
\newtheorem*{nota}{Notation}

\def\co{\colon\thinspace}

\newcommand{\AAA}{\mathcal A}

\newcommand{\BB}{\mathcal B}

\newcommand{\C}{\mathbb C}
\newcommand{\CC}{\mathcal C}

\newcommand{\D}{\mathbb D}
\newcommand{\DD}{\mathcal D}

\newcommand{\Hp}{\mathbb H}
\newcommand{\HH}{\mathcal H}

\newcommand{\MM}{\mathcal M}

\newcommand{\N}{\mathbb N}

\newcommand{\PP}{\mathcal P}

\newcommand{\R}{\mathbb R}

\newcommand{\UU}{\mathcal U}

\newcommand{\VV}{\mathcal V}

\newcommand{\WW}{\mathcal W}

\newcommand{\Z}{\mathbb Z}
\newcommand{\oz}{\overline{z}}

\newcommand{\DB}{\bar{\partial}}
\newcommand{\CR}[1]{\bar{\partial}_{#1}}

\newcommand{\lra}{\longrightarrow}
\newcommand{\ra}{\rightarrow}

\DeclareMathOperator{\aut}{\mathfrak{aut}}

\DeclareMathOperator{\crit}{\mathrm{Crit}}

\DeclareMathOperator{\Diff}{\mathrm{Diff}}
\DeclareMathOperator{\dist}{\mathrm{dist}}

\DeclareMathOperator{\ev}{\mathrm{ev}}

\DeclareMathOperator{\id}{\mathrm{id}}
\DeclareMathOperator{\im}{\mathrm{Im}}
\DeclareMathOperator{\Int}{\mathrm{Int}}

\DeclareMathOperator{\loc}{\mathrm{loc}}

\DeclareMathOperator{\ord}{\mathrm{ord}}

\DeclareMathOperator{\st}{\mathrm{st}}

\begin{document}

\author{Hansj\"org Geiges}
\author{Kai Zehmisch}
\address{Mathematisches Institut, Universit\"at zu K\"oln,
Weyertal 86--90, 50931 K\"oln, Germany}
\email{geiges@math.uni-koeln.de, kai.zehmisch@math.uni-koeln.de}

\title[Cerf's Theorem]{
  Eliashberg's Proof of Cerf's Theorem
}

\date{}

\begin{abstract}
Following a line of reasoning suggested by Eliashberg,
we prove Cerf's theorem that any diffeomorphism of the
$3$-sphere extends over the $4$-ball.
To this end we develop a moduli-theoretic version
of Eliashberg's filling-with-holomorphic-discs method.
\end{abstract}

\maketitle


\section{Introduction\label{intro}}
The abelian group $\Gamma_n$ of orientation preserving diffeomorphisms
of the $(n-1)$-dimensional sphere $S^{n-1}$ modulo those that
extend to a diffeomorphism of the $n$-ball $D^n$ plays an important role
in differential topology, cf.~\cite{kosi93}. By the classical
work of Kervaire--Milnor~\cite{kemi63}
on homotopy spheres and Smale's~\cite{smal61} solution
of the higher-dimensional Poincar\'e conjecture, $\Gamma_n$ can
be identified with the set of oriented smooth structures on the topological
$n$-sphere for $n\geq 5$. The correspondence is given by associating
with $[f]\in\Gamma_n$ the smooth structure on $S^n$ obtained by
using the diffeomorphism $f$ of $S^{n-1}$ to glue two copies of $D^n$
along their boundary.

It is easy to see that $\Gamma_1$ and $\Gamma_2$ are trivial.
The result $\Gamma_3=0$ is due independently to Munkres~\cite{munk60}
and Smale~\cite{smal59}. The argument of
Munkres is quite elementary; using the Poincar\'e--Bendixon
theorem for foliations of the plane, Smale actually proves the stronger
statement that the $3$-dimensional orthogonal group is a strong
deformation retract of the diffeomorphism group of the $2$-sphere.
For $n\geq 5$, the groups $\Gamma_n$
are amenable to computation by the results of Kervaire--Milnor,
for instance $\Gamma_5=\Gamma_6=0$, and $\Gamma_7$ is the cyclic
group of order~$28$.

The statement $\Gamma_4=0$ is known as {\em Cerf's theorem}~\cite{cerf68}.
One consequence of this result is that there are no
exotic smooth structures on $S^4$ that can be obtained by
gluing two $4$-discs. Thanks to the obstruction theory
of Munkres~\cite{munk64}
for smoothings of combinatorial manifolds, Cerf's theorem
also implies, for instance, that every combinatorial $n$-manifold
admits a smoothing for $n\leq 7$, unique up to
diffeomorphism for~$n\leq 6$.

In \cite{elia92} Eliashberg proposed an ingenious proof of Cerf's
theorem based on his classification of contact structures on~$S^3$
and his method of filling with holomorphic discs~\cite{elia90}.
Eliashberg--Polterovich~\cite{elpo96} gave another topological
application of this method; they showed that the space of local
Lagrangian $2$-knots in $\R^4$ is contractible.

Unfortunately, the survey paper \cite{elia90} does not include proofs
of the essential compactness and regularity properties
of holomorphic discs.
In the present paper we give a proof of Cerf's theorem along
the lines suggested by Eliashberg. Rather than simply
filling in the details, we develop an alternative
approach to the filling with holomorphic discs in a moduli-theoretic
framework, for which the technical foundations
have been laid in the magisterial
monograph of McDuff--Salamon~\cite{mcsa04}.

Here is an outline of the paper.
In Section~\ref{section:idea} we describe the basic set-up of Eliashberg's
proof, and then indicate where our strategy differs from Eliashberg's.
In Section~\ref{section:contact} we recall Eliashberg's
argument why it suffices to prove the extension result for
contactomorphisms of the standard contact structure $\xi$ on~$S^3$.

In Section~\ref{section:Bishop} we introduce the moduli space of
holomorphic discs on which our proof of Cerf's theorem is built.
The boundaries of the holomorphic discs in question are
required to lie in a family
of totally real submanifolds; in a different
setting such varying boundary
conditions have also been investigated by Wendl~\cite{wend05}.
As described in Section~\ref{section:extension}, a
suitable evaluation map on this moduli space then gives an
`explicit' extension of a given contactomorphism of $(S^3,\xi )$
to a diffeomorphism of~$D^4$.
Sections~\ref{section:compactness} and~\ref{section:transversality}
are devoted to establishing compactness and transversality results
for our moduli space.

In order to ascertain that the holomorphic discs in question
are embedded, we need a relative adjunction inequality,
which in turn relies on results about positivity of intersection.
These are demonstrated in Section~\ref{section:positivity},
preceded by an account of the topological intersection
theory of discs in Section~\ref{section:intersection}.

One important word about notation: following \cite{mcsa04},
and with apologies to hyperbolic geometers, we write $\D\subset\C$ and
$\Hp\subset\C$ for the {\em closed\/} unit disc and upper half-plane,
respectively.
\section{Idea of the proof}
\label{section:idea}
Regard the $3$-sphere $S^3$ as the unit sphere in~$\C^2$ with 
complex Cartesian coordinates $(z_1=x_1+iy_1,z_2=x_2+iy_2)$.
Let $H$ be the height function on $S^3$ given by projection onto
the $y_2$-coordinate. For $t\in(-1,1)$ the level sets
$S^t:= H^{-1}(t)$ define a smooth foliation of
$S^3\setminus\{(0,0,0,\pm1)\}$ by $2$-spheres.
We regard the points
\[ q^t_{\pm}:= (0,0,\pm \sqrt{1-t^2},t) \]
as the poles of these $2$-spheres.

This family of poles, together with the two poles
$(0,0,0,\pm 1)$ of $S^3$, forms an unknot
\[ K:= \bigl\{ (0,0,\pm\sqrt{1-t^2},t) \co t\in[-1,1]\bigr\}\]
in~$S^3$. The complement $S^3\setminus K$ is foliated
by circles that bound holomorphic discs
\[ D_s^t:=D^4\cap \bigl(\C\times\{ x_2=s,y_2=t\}\bigr),\;\;
|t|<1,\;\; |s|<\sqrt{1-t^2}.\]
For each $t\in (-1,1)$, the circles $\partial D_s^t$
foliate the punctured $2$-sphere $S^t\setminus\{ q^t_{\pm}\}$;
we write $\mathfrak{C}_t$ for this foliation by circles.

Now let $\varphi$ be an orientation preserving diffeomorphism
of~$S^3$ that we wish to extend over~$D^4$. Any
isotopy from $\varphi$ to a diffeomorphism $\psi$ can be
swept out over a collar neighbourhood of $S^3$ in~$D^4$, where
the collar coordinate serves as isotopy parameter. Thus, if
$\psi$ extends over~$D^4$, then so does~$\varphi$.
In other words, it suffices to find an extension for a suitably
well-behaved representative in the isotopy class of~$\varphi$.

By the disc theorem~\cite[Theorem~8.3.1]{hirs76}, any
orientation preserving diffeomorphism of an $n$-manifold is isotopic to
one that fixes any given embedded $n$-disc. So we may require
without loss of generality that $\varphi$ fix a neighbourhood
of the unknot~$K$.

Write $\tilde{S}^t:=\varphi (S^t)$ for the images of the
$2$-spheres $S^t$ under~$\varphi$. Then $\varphi (\mathfrak{C}_t)$
is a foliation of the punctured $2$-sphere
$\tilde{S}^t\setminus\{ q^t_{\pm}\}$ by circles. Now the key idea
is as follows. Suppose we were able to find a foliation of
$D^4\setminus K$ by holomorphic discs with these properties:
\begin{itemize}
\item[(1)] The boundaries of the discs foliate the punctured $2$-spheres
$\tilde{S}^t\setminus\{ q^t_{\pm}\}$.
\item[(2)] This foliation of $\tilde{S}^t\setminus\{ q^t_{\pm}\}$
can be deformed to $\varphi (\mathfrak{C}_t)$.
\end{itemize}

Then the proof of Cerf's theorem $\Gamma_4=0$ would reduce
to $\Gamma_2=0$, or more precisely the
parametric version of the latter, which says that
the restriction map $\Diff (D^2)\ra\Diff (S^1)$ is
a Serre fibration with contractible fibre.

As observed by Eliashberg, this plan is feasible,
provided that $\varphi$ is a contactomorphism of the standard
contact structure $\xi$ on~$S^3$. Then both $\varphi (\mathfrak{C}_t)$
and the foliation of
$\tilde{S}^t\setminus\{ q^t_{\pm}\}$ given by the boundaries
of holomorphic discs are transverse to the characteristic foliation
$\tilde{S}^t_{\xi}$, so one foliation can be deformed into the other
by an isotopy along the characteristic foliation.
Moreover, Eliashberg's classification of contact
structures on $S^3$ shows that any diffeomorphism of $S^3$ is
indeed isotopic to a contactomorphism of $(S^3,\xi)$ --- for the
convenience of the
reader we reproduce the proof of this fact in
Section~\ref{section:contact} ---, so this assumption on
$\varphi$ is not restrictive.

In the present paper, Eliashberg's programme is carried out rigorously.
However, we deviate from his scheme in one important
respect. We define a moduli space of holomorphic discs
in $D^4\subset\C^2$ whose boundaries lie on the punctured spheres
$\tilde{S}^t\setminus\{ q^t_{\pm}\}$. The extension of
the contactomorphism $\varphi$ over $D^4$ is then defined
in terms of an evaluation map on this moduli space.
That way we obtain easier control over
the smoothness of the extension. This is true in particular
near the singular set~$K$, where the behaviour of the diffeomorphism
is clear from the observation that
the filling with holomorphic discs adapted to $\varphi$
coincides with the standard filling near~$K$.

The reader will notice that the vanishing of $\Gamma_2$ is not mentioned
explicitly in our proof. This is a consequence of
our working with holomorphic discs with three marked points on the boundary.
Such discs come with a fixed parametrisation, so we need no longer
worry about choices of diffeomorphisms of $S^1$ and their
extensions to~$D^2$.

In the following subsections we describe the backcloth of our proof,
mostly to set up notation.
\subsection{Contact geometry of~$S^3$}
\label{subsection:contact}
Consider the differential $1$-form
\[ \lambda := \frac{1}{2} (x_1\, dy_1-y_1\, dx_1+x_2\, dy_2-y_2\, dx_2)\]
on~$\R^4$. The $1$-form $\alpha :=\lambda|_{TS^3}$ on $S^3$
gives rise to a volume form $\alpha\wedge d\alpha$ defining
the positive orientation of~$S^3$ (with $S^3$
oriented as boundary of $D^4\subset\C^2$), so $\alpha$ is
a {\bf contact form}. Its kernel $\xi :=\ker\alpha$ is
the (positive) {\bf standard contact structure} on~$S^3$.
The $2$-form $\omega:=d\lambda$ is the standard symplectic form on~$\R^4$.

The {\bf characteristic foliation} $S^t_{\xi}$ induced by $\xi$
on the $2$-sphere $S^t$ is the singular $1$-dimensional
foliation defined by the intersection $T_pS^t\cap\xi_p$ for $p\in S^t$,
with singular points where $T_pS^t$ and the contact plane $\xi_p$ coincide,
which happens exactly at $p=q^t_{\pm}$. Notice that the unknot $K$
made up of these singular points and the two poles of $S^3$
is transverse to~$\xi$.

A {\bf contactomorphism} of $\xi=\ker\alpha$ is a diffeomorphism
$\varphi$ of $S^3$ such that $\varphi^*\alpha=f\alpha$ for some
smooth function $f\co S^3\ra\R^+$.

In Section~\ref{section:contact} below we also have to use the notion
of a tight contact structure, and we appeal to Gray stability
of contact structures. Both concepts are explained in~\cite{geig08}.
\subsection{Strict pseudoconvexity of~$S^3$}
\label{subsection:strict}
Write $J_0$ for the standard complex bundle structure on
$T\C^2$ induced by multiplication with~$i$.
The $3$-sphere $S^3$ may be regarded as the level set $\rho^{-1}(1)$
of the strictly plurisubharmonic function
\[\rho\co\C^2\lra\R,\;\;(z_1,z_2)\longmapsto |z_1|^2+|z_2|^2.\]
This implies that the complex tangencies $T_pS^3\cap J_0(T_pS^3)$
define a contact structure on~$S^3$, given as the kernel of
the contact form $-(1/2)d\rho\circ J_0|_{TS^3}$,
cf.~\cite[Chapter~5]{geig08}. As is well known (and easy to check),
this contact form equals the previously defined~$\alpha$.

The essential consequence of $S^3$ being strictly pseudoconvex is
that non-constant holomorphic discs in $\C^2$ with boundary on $S^3$
have their interior in the interior of~$D^4$, see
Proposition~\ref{prop:maximum} for the precise formulation of this
maximum principle. 

The characteristic foliation $S^t_{\xi}$ being non-singular
away from $q^t_{\pm}$ is equivalent to $S^t$ not having
any complex tangencies except at these two poles. In other
words, $S^t\setminus\{ q^t_{\pm}\}$ is a totally real surface.
\subsection{A holomorphic filling of~$S^3$}
For $|t|<1$ and $|s|<\sqrt{1-t^2}$ define a smooth 
real-valued function
\[ \theta (s,t):= \frac{t}{2\sqrt{1-t^2}}\cdot\ln
\left( \frac{\sqrt{1-t^2}+s}{\sqrt{1-t^2}-s} \right).\]
For each $t$ this defines a diffeomorphism from
$(-\sqrt{1-t^2},\sqrt{1-t^2})$ to~$\R$. Now consider
the parametrisations
\[ u^t_s(z):= \bigl( \sqrt{1-s^2-t^2}\cdot e^{i\theta (s,t)}\cdot z,
                   s,t\bigl),\;\; z\in\D,\]
of the holomorphic discs $D^t_s$. The rotation factor $e^{i\theta (s,t)}$
has been chosen in such a way that each leaf of the characteristic
foliation $S^t_{\xi}$ (outside $q^t_{\pm}$) is parametrised by
a map $s\mapsto u^t_s(z)$, $|s|<\sqrt{1-t^2}$, for some $z\in\partial\D$.
For this one simply needs to verify that
the tangent vector $\partial_s u^t_s(z)$ to~$S^3$
(for~$z\in\partial\D$) lies in the kernel of the $1$-form
$x_1\, dy_1-y_1\, dx_1-t\, dx_2$, which is the pull-back
of $\alpha$ to~$S^t$. 
Putting all these parametrisations together, we obtain a
diffeomorphism
\[ \begin{array}{rccc}
F_{\st}\co & (\D\times\Int\D,\partial\D\times\Int\D) & \lra        &
        (D^4\setminus K,S^3\setminus K)\\
           & (z,s,t)                                 & \longmapsto &
        u^t_s(z).
\end{array} \]
We can extend this map continuously over the boundary $\D\times\partial\D$
by sending $(z,s,t)$ with $|t|<1$ and $s=\pm \sqrt{1-t^2}$ to~$q^t_{\pm}$,
and $(z,0,\pm 1)$ to $(0,0,0,\pm 1)$. So the boundary is mapped onto~$K$.

This map $F_{\st}$ will be our prototype of a holomorphic filling.
The formal definition of such fillings
will be given in Section~\ref{section:extension}.
\subsection{Symplectic energy}
The {\bf symplectic energy} of the holomorphic disc $u^t_s$ is
defined as $E(u^t_s)=\int_{\D}(u^t_s)^*\omega$. One computes
\[ E(u^t_s)=\pi (1-s^2-t^2)\leq \pi (1-t^2),\]
i.e.\ for each $t$ we have a uniform bound on the energy.
A general result to this effect will be proved in
Proposition~\ref{prop:energy}.

Writing $z$ in polar coordinates as $z=re^{i\theta}$,
we compute
\[ \alpha (\partial_{\theta}u^t_s(z))=(1-s^2-t^2)/2\;\;\mbox{\rm for}
\;\; z\in\partial\D.\]
This gives another way to verify
the computation of the energy, since $\int_{\D}(u^t_s)^*\omega
=\int_{\partial\D}(u^t_s)^*\alpha$ by Stokes's theorem.

Since $\alpha$ evaluates positively on $\partial_{\theta}u^t_s(z)$
for $z\in\partial\D$,
the curve $u^t_s|_{\partial\D}$ is positively
transverse to the characteristic
foliation~$S^t_{\xi}$. It is another consequence of the maximum principle
that this holds true for arbitrary holomorphic discs,
see Proposition~\ref{prop:maximum}.
\subsection{Bundle pairs and Maslov index}
\label{subsection:Maslov}
The holomorphic disc $u^t_s$ has boundary on the totally real
submanifold $S^t\setminus\{ q^t_{\pm}\}$ of~$\C^2$.
So it defines a {\bf bundle pair} in the sense of
\cite[Definition~C.3.4]{mcsa04}, that is, a complex vector bundle
$(u^t_s)^*T\C^2$ over~$\D$ (here: the trivial $\C^2$-bundle) and a totally
real subbundle $(u^t_s|_{\partial\D})^*TS^t$
over the boundary~$\partial\D$.
At the point $u_0^t(e^{i\theta})
\in\partial D^t_0$, the fibre of this totally real subbundle
is given by $\R ie^{i\theta}\oplus\R$. From the axiomatic
definition of the Maslov index~$\mu$ given
in~\cite[Theorem~C.3.5]{mcsa04} it follows that $\mu (u^t_s)=2$
for all $t\in (-1,1)$ and $|s|<\sqrt{1-t^2}$.
\subsection{The standard neighbourhood of~$K$}
To avoid problems at the singular points~$q^t_{\pm}$, we
design a set-up where the holomorphic filling of $D^4$ corresponding to
some contactomorphism of $S^3$ coincides with the standard filling
$F_{\st}$ near~$K$.
For $\delta\in (0,1)$, we define a neighbourhood of $K$ by
\[ \UU^{\delta}:=F_{\st}\bigl(S^1\times\{ 1-\delta < s^2+t^2\leq 1\}\bigr)
\subset S^3.\]
The boundary $\partial\UU^{\delta}$ is a Lagrangian torus
with a holomorphic filling given by the restriction of $F_{\st}$
to $\D\times\{ s^2+t^2=1-\delta\}$.
We shall require the contactomorphism of $(S^3,\xi )$ to act as the
identity on some~$\UU^{\delta}$.
\subsection{The $2$-discs $Q_k$}
\label{subsection:Q}
In order to obtain a compact moduli space of holomorphic discs,
one needs to quotient out the $3$-dimensional automorphism
group of~$\D$ or, alternatively, place a restriction on three marked
points in~$\D$. We shall take the latter approach, using
the three points $i^k\in\partial\D$, $k=0,1,2$ as markers.

We define three open $2$-discs in $S^3$ by
\[ Q_k:=F_{\st}\bigl(\{i^k\}\times\Int\D\bigr)\subset S^3.\]
These discs are transverse to the punctured $2$-spheres
$S^t\setminus\{ q^t_{\pm}\}$, and the intersection
is given by one leaf $\ell^t_k$ of the characteristic foliation~$S^t_{\xi}$.

Notice that a contactomorphism $\varphi$ of $(S^3,\xi )$ will
map the characteristic foliation $S^t_{\xi}$ to the
characteristic foliation of the image sphere $\tilde{S}^t=\varphi (S^t)$.
The discs $Q_k$ will be used in Lemma~\ref{lem:ev} to show that,
as expected, imposing the condition that our holomorphic discs
map $i^k$ into the leaf $\tilde{\ell}^t_k:=\varphi (\ell^t_k)$
will cut down the dimensions of the moduli space by~$3$.

We also consider the smaller closed disc
\[ Q^{\delta}:=F_{\st}\bigl(\{1\}\times\{s^2+t^2\leq 1-\delta\}\bigr)
\subset Q_0.\]
In Proposition~\ref{prop:ev} we shall set up a diffeomorphism
between the disc $\varphi (Q^{\delta})$ and the moduli space
of holomorphic discs adapted to~$\varphi$, with three marked points
and boundary outside the neighbourhood~$\UU^{\delta}$.
\section{Reduction to contact geometry}
\label{section:contact}
In order to construct an extension of a given diffeomorphism of $S^3$ to a
diffeomorphism of~$D^4$, we first want to isotope it
to a diffeomorphism adapted to the set-up of the
previous section.

\begin{prop}
\label{prop:contact}
Any orientation preserving diffeomorphism of $S^3$
is isotopic to a contactomorphism of $(S^3,\xi)$
that fixes an open neighbourhood $\UU$ of $K$ pointwise.
\end{prop}

\begin{proof}
Let $\varphi_0$ be a given orientation preserving
diffeomorphism of~$S^3$. Then $T\varphi (\xi )$ is a positive
tight contact structure on the $3$-sphere. By the uniqueness of
such structures up to isotopy, and with Gray stability, we can isotope
$\varphi_0$ to a contactomorphism $\varphi_1$
of~$\xi$, cf.~\cite[Lemma~4.11.1]{geig08}.

There is a contactomorphism
\[ (\R^3,\ker (dw+u\, dv -v\, du))\lra
(S^3\setminus\{ (0,1,0,0)\} ,\xi )\]
which sends the
unit circle $C$ in the $uv$-plane to~$K$. An explicit description
of such a contactomorphism can be found
in~\cite[Proposition~2.1.8]{geig08}; simply compose
the contactomorphism
\[ (\R^3,\ker (dw+u\, dv -v\, du))\lra
(S^3\setminus\{ (0,0,0,1)\},\xi )\]
described there with the
contactomorphism of $(S^3,\xi )$ induced by $(z_1,z_2)\mapsto (z_2,z_1)$.

This contactomorphism restricts to a contact embedding $j$ of
a ball $B$ of radius~$2$, say, in $(\R^3,\ker (dw+u\, dv -v\, du))$
into $(S^3,\xi )$, sending $C$ to~$K$. A second contact
embedding of this ball is given by $\varphi_1\circ j$.
By the contact disc theorem~\cite[Theorem 2.6.7 and Remark~2.6.8]{geig08},
there is a contact isotopy whose time-$1$ map $\psi_1$
is a contactomorphism of $(S^3,\xi )$ such that
$\psi_1\circ\varphi_1\circ j=j$.

So our initial diffeomorphism $\varphi_0$ is isotopic to the contactomorphism
$\psi_1\circ\varphi_1$, which fixes the neighbourhood
$\UU:=j(\mathrm{int}\, B)$.
\end{proof}

The isotopy from $\varphi_0$ to $\psi_1\circ\varphi_1$
can be swept out over a collar neighbourhood of $S^3=\partial D^4$
in~$D^4$. Hence, Cerf's theorem will follow if we can find
an extension of $\psi_1\circ\varphi_1$ to a diffeomorphism of~$D^4$.

So from now on we consider a contactomorphism $\varphi$ of $(S^3,\xi)$
that fixes an open neighbourhood $\UU$ of $K$ pointwise,
and we choose $\delta\in(0,1)$ such that the 
closure $\overline{\UU}^{\delta}$ of the standard
neighbourhood $\UU^{\delta}$ is contained in~$\UU$.
\section{Bishop discs}
\label{section:Bishop}
In Section \ref{section:idea} we described a simple filling
of the $3$-sphere by holomorphic discs, i.e.\ a foliation
of $D^4\setminus K$ by discs with boundary in $S^3\setminus K$.
We now want to construct another such filling, one that is
related to the contactomorphism~$\varphi$.

We begin by introducing the corresponding boundary value problem.
For fixed $t\in(-1,1)$ we are looking for smooth (up to the boundary)
solutions
\[ u^t\co (\D,\partial\D)\lra (D^4,\tilde{S}^t\setminus\{q^t_{\pm}\})\]
of the homogeneous Cauchy--Riemann equation $\DB u^t=0$
(in other words: holomorphic discs)
having boundary values in $\tilde{S}^t:=\varphi(S^t)$.
We remarked before that the punctured $2$-spheres
$S^t\setminus\{ q^t_{\pm}\}$
are totally real submanifolds of~$\C^2$. Hence, so are their
images $\tilde{S}^t\setminus\{q^t_{\pm}\}$ under $\varphi$,
since $\varphi$ preserves the contact structure
$\xi =TS^3\cap J_0(TS^3)$.

In addition, we require the relative homotopy class
$[u^t]\in \pi_2(D^4,\tilde{S}^t\setminus\{q^t_{\pm}\})$
to equal $A^t$, which we define as the class
in $\pi_2(D^4,\tilde{S}^t\setminus\{q^t_{\pm}\})$
that maps to the class of $[\varphi\circ u^t_0|_{\partial\D}]\in
\pi_1(\tilde{S}^t\setminus\{q^t_{\pm}\})$ under the boundary
homomorphism~$\partial_*$. (This notion of relative homotopy class
will be relevant for our discussion of Maslov indices, and our
more general set-up in Section~\ref{section:positivity}.)
Moreover, we fix three marked points by imposing the condition
$u^t(i^k)\in\tilde{\ell}^t_k$, where
$\tilde{\ell}^t_k:=\varphi(\ell^t_k)$, $k=0,1,2$, are three leaves
of the characteristic foliation $\tilde{S}^t_{\xi}=\varphi(S^t_{\xi})$.
Such a holomorphic disc will be called a
$t$-{\bf level Bishop disc for~$\varphi$}.
The collection
\[ \WW_{\varphi}:= \bigl\{ u^t\co t\in (-1,1),\;
u^t\;\text{is a}\; t \text{-level Bishop disc for} \; \varphi \bigr\}\]
of all such discs is the {\bf moduli space of Bishop discs}.
We write $\MM_{\varphi}(t)\subset\WW_{\varphi}$ for the moduli space
of $t$-level Bishop discs for~$\varphi$.

Section~\ref{section:idea} can be read as a description
of the structure of $\WW_{\id}$.
We are now aiming for a similar description of $\WW_{\varphi}$
for any contactomorphism $\varphi$ as in Proposition~\ref{prop:contact}.

In the following propositions we collect some relevant facts about our
Bishop discs.

\begin{prop}
\label{prop:mu2}
Every Bishop discs for $\varphi$ has Maslov index~$2$, that is,
$\mu (A^t)=2$ for all $t\in (-1,1)$.
\end{prop}

\begin{proof}
Recall from Section~\ref{subsection:Maslov} that
that the Maslov index of $u^t_0$ equals $2$ for all $t\in(-1,1)$.
Since the boundary homomorphism $\partial_*\co
\pi_2(D^4,\tilde{S}^t\setminus\{q^t_{\pm}\})\ra
\pi_1(\tilde{S}^t\setminus\{q^t_{\pm}\})$ is an isomorphism,
and the Maslov index is invariant under homotopies, it suffices to
show that $\partial_*A^t=\partial_*[u^t_0]$.

For a given $t$ choose $s\in (-1,1)$ with $1-\delta-t^2<s^2<1-t^2$.
Then, since $\varphi =\id$ on $\UU^{\delta}$,
we have $[\varphi\circ u^t_s|_{\partial\D}]=
[u^t_s|_{\partial\D}]$. Hence, we obtain the following sequence of
equalities, where the first one holds by definition of the
class~$A^t$.
\[ \partial_*A^t=[\varphi\circ u^t_0|_{\partial\D}]=[\varphi\circ
u^t_s|_{\partial\D}]=[u^t_s|_{\partial\D}]=[u^t_0|_{\partial\D}]=
\partial_*[u^t_0].\]
This proves the proposition.
\end{proof}

In the sequel we have to appeal several times to the maximum principle
for holomorphic discs in $\C^2$ with boundary on~$S^3$. We briefly recall
the statement in the form needed for our purposes. This maximum principle
also explains why we could define our Bishop disc from the outset
as maps into the $4$-ball $D^4$ rather than~$\C^2$.

\begin{prop}[Maximum principle]
\label{prop:maximum}
Let $\Sigma\subset S^3$ be a totally real surface in~$\C^2$, so that
the characteristic foliation $\Sigma_{\xi}$ is non-singular.
Let $u\co (\D,\partial\D )\ra (\C^2,\Sigma )$ be a non-constant holomorphic
disc. Then $u$ maps the interior of $\D$ to the interior of the
$4$-ball~$D^4$, and $u|_{\partial\D}$ is an immersion
positively transverse to~$\Sigma_{\xi}$.
\end{prop}

\begin{proof}
A straightforward computation, cf.~\cite[Lemma~4.11.3]{geig08},
gives
\[ \Delta (\rho\circ u)=2\, d\lambda (u_x,J_0u_x)\geq 0.\]
For the definition of $\lambda$ and $\rho$ see Sections
\ref{subsection:contact} and~\ref{subsection:strict}, respectively.
Then the strong maximum principle of E.~Hopf, cf.\
\cite[Section~2.3]{prwe84}, shows that $\rho\circ u=\|u\|^2$ cannot attain
a maximum in the interior of~$\D$. The boundary point lemma ({\em ibid.})
implies that $u$ is transverse to $S^3$ along the boundary.
Since the characteristic foliation $\Sigma_{\xi}$ is given by
the intersection of the complex tangencies $\xi$ to $S^3$ with~$\Sigma$,
and $u$ is holomorphic, this forces $u|_{\partial\D}$ to be an
immersion transverse to~$\Sigma_{\xi}$. The statement about this
immersion being {\em positively\/} transverse to $\Sigma_{\xi}$
follows because the complex structure $J_0$ sends the outer
normal of $S^3\subset\C^2$ to a vector positively transverse
to $\xi=\ker\alpha$.
\end{proof}

\begin{rem}
Suppose $M$ is a compact level set $M=\rho^{-1}(1)$ of a smooth function
$\rho\co W\rightarrow\R$ on an almost complex manifold $(W,J)$,
where $1$ and nearby values are regular for~$\rho$; any compact
orientable codimension~$1$ submanifold of $W$ can be realised in this way.
The level set $M$ is called {\bf $J$-convex} if the complex tangencies
to $M$ define a positive contact structure, that is, if 
$-d\rho\circ J|_{TM}$ is a contact form inducing
the natural orientation of~$M$. One then computes that
for a smooth function $f\co\R\ra\R$ with $f''\gg f'>0$ the composition
$f\circ\rho$ is strictly plurisubharmonic, i.e.\
$-d(d(f\circ\rho)\circ J)(X,JX)>0$ for all nonzero tangent vectors
$X\in TW$ in a neighbourhood of~$M$. From this observation
one concludes that the maximum principle holds, {\em mutatis mutandis},
for $J$-convex hypersurfaces.
\end{rem}

Following Lazzarini~\cite{lazz00}, cf.~\cite[Appendix~E]{mcsa04},
we say a holomorphic disc $u$ is {\bf simple}
if there are no two disjoint non-empty open subsets $U,V\subset\D$
such that $u(U)=u(V)$.

\begin{prop}
\label{prop:simple}
All Bishop discs are simple.
\end{prop}

\begin{proof}
The maximum principle tells us that $u|_{\partial\D}\co\partial\D\ra
\tilde{S}^t\setminus\{ q^t_{\pm}\}$ is an immersion transverse
to the characteristic foliation $\tilde{S}^t_{\xi}$.
By the homotopical condition on $u|_{\partial\D}$
and the nature of the characteristic foliation $\tilde{S}^t_{\xi}$,
the restriction of $u$ to $\partial\D$ must then in fact be an embedding.

From this observation one concludes that $u^t$ must be simple; since the
argument for this last step is of interest elsewhere, we formulate it
(in slightly greater generality) as a separate lemma.
\end{proof}

\begin{lem}
\label{lem:simple}
Let $(W,J)$ be an almost complex manifold.
If $u\co \D\ra W$
is a $J$-holomorphic disc such that $u|_{\partial\D}$ is
an embedding and no interior point of $\D$ is mapped to~$u(\partial\D )$,
then $u$ is simple.
\end{lem}

\begin{rem}
In the case of our Bishop discs, $u|_{\partial\D}$
is an embedding into the strictly pseudoconvex boundary $S^3$ of~$D^4$;
the condition on the interior points of $\D$ is then guaranteed by
the maximum principle.
\end{rem}

\begin{proof}[Proof of Lemma~\ref{lem:simple}]
Arguing by contradiction, we assume that there are two
disjoint non-empty open subsets $U,V\subset\D$ with $u(U)=u(V)$.
By the conditions of the lemma, $U$ and $V$ are disjoint from~$\partial\D$.
In the proof of \cite[Lemma~2.4.1]{mcsa04} it is shown that 
critical points of $u$ in the interior of $\D$ are isolated.
Since $u|_{\partial\D}$ is an embedding, critical points
cannot accumulate near $\partial\D$, so there are
only finitely many of them.
Choose a non-critical point $z^*\in U$ and a half-open
line segment $[z^*,z^{**})$
from $z^*$ to a point $z^{**}\in\partial\D$, disjoint
from the critical points of~$u$. Let $P\subset [z^*,z^{**})$
be the set of points that have a neighbourhood in $[z^*,z^{**})$
consisting of non-injective points for~$u$, that is, points $z$
for which there is a different point $z'\in\D$ with $u(z)=u(z')$.
By definition $P$ is open in $[z^*,z^{**})$, and $z^*\in P$,
so $P$ is non-empty.

We claim that $P$ is also closed in $[z^*,z^{**})$. Indeed,
suppose that $(z_{\nu})$ is a sequence of points in $P$
converging to some point~$z_0\in [z^*,z^{**})$. Then there are points
$w_{\nu}\neq z_{\nu}$ with $u(w_{\nu})=u(z_{\nu})$. By passing to a
subsequence we may assume that the sequence $(w_{\nu})$ converges to a
point $w_0\in\D$.

Since the point $z_0$ is not
critical for~$u$, the map $u$ is locally injective near~$z_0$.
This implies that $w_0\neq z_0$, and
without loss of generality we may assume that $w_{\nu}\neq z_0$
for all~$\nu$. Moreover, the assumptions of the lemma imply that
$w_0$ is an interior point of~$\D$.

With this information we are precisely in the situation of
\cite[Lemma~2.4.3]{mcsa04}, which tells us that there exists
a holomorphic map $\phi$, defined in a neighbourhood of~$z_0$,
such that $\phi (z_0)=w_0$ and $u=u\circ\phi$. So $z_0$
has a neighbourhood of non-injective points, which means that
$z_0\in P$, i.e.\ $P$ is closed in $[z^*,z^{**})$.

We conclude that $P=[z^*,z^{**})$. In particular, we find a sequence of
non-injective points, which we write again as $(z_{\nu})$,
accumulating at~$z^{**}$. So as before we have points
$w_{\nu}\neq z_{\nu}$ with $u(w_{\nu})=u(z_{\nu})$, with
$(w_{\nu})$ converging to some point~$w_0$. Then $u(w_0)=u(z^{**})$,
which implies $w_0=z^{**}$. But $u$ is locally injective
near~$z^{**}$, so it is impossible for the two sequences
$(z_{\nu})$ and $(w_{\nu})$ of corresponding non-injective points
to accumulate at~$z^{**}$. This contradiction proves the lemma.
\end{proof}

\begin{rem}
Instead of asking for $u|_{\partial\D}$ to be an embedding, it
suffices to require that $u|_{\partial\D}$ be injective. Then the
set of critical points is still finite, see~\cite[Theorem~3.5]{lazz00},
and the proof above goes through {\em verbatim}.
\end{rem}

\begin{prop}
\label{prop:embedded}
All Bishop discs are embedded and mutually disjoint.
\end{prop}

\begin{proof}
Since all Bishop discs are simple, the results from
Section~\ref{section:positivity} below on the positivity of intersections
apply. In that section we define a so-called
embedding defect $D$ for holomorphic discs.
This embedding defect depends only
on the relative homotopy class $A^t=[u^t]\in\pi_2(D^4,\tilde{S}^t
\setminus\{ q^t_{\pm}\} )$, and it equals zero if and only
if $u^t$ is embedded (Theorem~\ref{thm:adjunction}).

For each $t\in (-1,1)$, the discs $u^t_s$ with 
$s$ in the range given by $1-\delta -t^2<s^2<1-t^2$
belong to our moduli space $\MM_{\varphi}(t)$ of $t$-level Bishop
discs. In particular, $\MM_{\varphi}(t)$ contains at
least two disjoint embedded holomorphic discs. It follows that
$D(A^t)=0$ and that all Bishop discs are embedded.

Given two Bishop discs of distinct levels, their boundary curves
are disjoint.  Then the intersection number of the two discs
equals the linking number of their boundary curves
in~$S^3$, which is zero. By positivity of intersections
(at interior points), see~\cite[Theorem~7.1]{miwh95},
it follows that the two discs must be disjoint.

It remains to show that the same holds true for Bishop discs
at one and the same level~$t$. From the defining equation
of the embedding defect in Section~\ref{section:positivity},
and with $\mu (A^t)=2$,
it follows that the self-intersection number $A^t\bullet A^t$,
as defined in Section~\ref{section:intersection}, equals zero for every~$t$.
(Alternatively, this is a consequence of the existence of two disjoint
$t$-level Bishop discs.) With positivity of intersections
(Theorem~\ref{thm:positivity})
we infer that any two $t$-level Bishop discs are either disjoint,
or their images coincide. Since we prescribed the
images of three marked points on each Bishop disc,
two Bishop discs with the same image are actually identical.
\end{proof}

\begin{cor}
\label{cor:standard}
If the boundary of a $t$-level Bishop disc $u$ hits the set
$\overline{\UU}^{\delta}$, then $u=u_s^t$ for some $s$
with $1-\delta-t^2\leq s^2<1-t^2$.
\end{cor}

\begin{proof}
The contactomorphism $\varphi$ is the identity on $\UU$,
which contains $\overline{\UU}^{\delta}$. That latter set
is foliated by the boundary circles of standard discs $u^t_s$.
Now invoke the preceding proposition.
\end{proof}

It is therefore opportune to restrict attention to the
{\bf truncated moduli space}
\begin{eqnarray*}
\lefteqn{\WW_{\varphi}^{\delta}:= \bigl\{ u^t \co
t\in [-\sqrt{1-\delta},\sqrt{1-\delta}],\bigr. }\\
& &  \bigl. u^t\;
\text{is a}\; t\text{-level Bishop disc for} \; \varphi \;
\text{such that} \; u^t(\partial\D)\subset S^3\setminus
\UU^{\delta} \bigr\} .
\end{eqnarray*}
In Sections \ref{section:compactness} and \ref{section:transversality}
we shall prove that this moduli space is a compact manifold with boundary.
Working with this truncated moduli space allows us to
circumvent all subtleties along the singular set~$K$.

The following uniform energy estimate will be one ingredient in
that compactness argument.

\begin{prop}
\label{prop:energy}
The symplectic energy of all Bishop discs for $\varphi$ is uniformly bounded
by a constant depending only on~$\varphi$.
\end{prop}

\begin{proof}
By Proposition~\ref{prop:embedded} any Bishop disc $u$ is an
embedding. So with Stokes's theorem and the transformation formula,
the energy $E(u)$ can be computed as
\[ E(u)= \int_{\D}u^*\omega= \int_{u(\partial\D)}\alpha
=\int_{\varphi^{-1}\circ u(\partial\D)}\varphi^*\alpha .\]
The contactomorphism $\varphi$ pulls back the contact form $\alpha$
to $f\alpha$, where $f\co S^3\ra\R^+$ is a smooth function
bounded above by some constant~$C=C(\varphi )$.
By the maximum principle,
$u|_{\partial\D}$ is positively transverse to~$\xi =\ker\alpha$,
and the same holds for $\varphi^{-1}\circ u|_{\partial\D}$.
It follows that
\[ E(u)\leq C\int_{\varphi^{-1}\circ u(\partial\D)}\alpha .\]

The $2$-sphere $S^t$, which contains the embedded circle
$\varphi^{-1}\circ u(\partial\D)$, is naturally oriented by the
area form
\[ \sigma^t:= \frac{1}{\sqrt{1-t^2}} \bigl( x_1\, dy_1\wedge dx_2+
 y_1\, dx_2\wedge dx_1+ x_2\, dx_1\wedge dy_1 \bigr) ;\]
the total area of $S^t$ equals $4\pi (1-t^2)$.
Let $D^t$ be the $2$-disc in $S^t$ (with the induced orientation)
whose oriented boundary equals $\varphi^{-1}\circ u(\partial\D)$. Then
\[ E(u)\leq C\int_{\partial D^t}\alpha= C\int_{D^t}\omega .\]

On $TS^t$ we have $x_1\, dx_1+y_1\, dy_1+x_2\, dx_2=0$. Using this,
one finds that
\[ x_2\sigma^t=\sqrt{1-t^2}\, dx_1\wedge dy_1=\sqrt{1-t^2}\, \omega\;\;
\mbox{\rm on}\;\; TS^t.\]
Hence, with $x_2\leq\sqrt{1-t^2}$ we conclude
\[ E(u)\leq \frac{C}{\sqrt{1-t^2}}\int_{D^t}x_2\sigma^t\leq
C\int_{S^t}\sigma^t= 4\pi C(1-t^2)\leq 4\pi C.\]
This is the desired uniform estimate.
\end{proof}

\begin{rem}
The symplectic energy $E(u)$ of a holomorphic disc
equals its {\bf Dirichlet energy}
$(1/2)\int_{\D}|\nabla u|^2$
see~\cite[Section~2.2]{mcsa04}. Since the symplectic energy is,
by its definition, invariant under reparametrisations,
it follows that the Dirichlet energy is invariant under
{\em conformal} reparametrisations.
\end{rem}
\section{From the filling to the extension}
\label{section:extension}
By taking the results on compactness and transversality from Sections
\ref{section:compactness} and \ref{section:transversality} for granted, we
can now determine the truncated moduli space.

Recall from Section~\ref{subsection:Q}
that $Q^{\delta}$ is a closed $2$-disc in $S^3$ defined
by $Q^{\delta}=\{ u_s^t(1)\co (s,t)\in\D_{1-\delta}\}$, where
$\D_{1-\delta}$ denotes the closed disc
$\{ s^2+t^2\leq 1-\delta\}$. Hence,
if $u$ is a Bishop disc for~$\varphi$, our condition on
the marked points implies $u(1)\in
\varphi (Q^{\delta})$.

\begin{prop}
\label{prop:ev}
The evaluation map
\[ \begin{array}{rccc}
\ev_1\co & \WW_{\varphi}^{\delta} & \lra        & 
                               \varphi ({Q}^{\delta})\\
         & u                      & \longmapsto &
                               u(1)
\end{array} \]
is a diffeomorphism.
\end{prop}

This proposition will be proved at the end of
Section~\ref{section:transversality},
once we have established the relevant compactness and transversality results.

We now want to define a notion of holomorphic
filling of $S^3$ (with singular set~$K$)
with respect to a contactomorphism
$\varphi$ of $(S^3,\xi )$ that subsumes, for the identity map,
the standard filling $F_{\st}$.
In the sequel, the contactomorphism $\varphi$ is always taken as a
given and will be omitted from the notation.

\begin{defn}
A {\bf holomorphic filling} of $S^3$ is a diffeomorphism
\[ F\co \bigl(\D\times \Int\D,\partial\D\times \Int\D\bigr) \lra
(D^4\setminus K,S^3\setminus K)\]
with the following properties:
\begin{itemize}
\item[(1)] $\DB F(\, .\, ,s,t)=0$ for all $(s,t)\in \Int\D$;
\item[(2)] $F$ extends continuously to a map (still denoted~$F$)
defined on $\D\times\D$ such that $F(\D\times\partial\D)=K$;
\item[(3)] for all $(z,s,t)\in\partial\D\times \Int\D$ we have
$F(z,s,t)\in\tilde{S}^t\setminus\{q^t_{\pm}\}$,
and as $s$ tends to $\pm\sqrt{1-t^2}$, the limit of $F(z,s,t)$
equals $q^t_{\pm}$ for each $z\in\partial\D$.
\end{itemize}
\end{defn}

Observe that for $t\in (-1,1)$, the
punctured $2$-sphere $\tilde{S}^t\setminus\{q^t_{\pm}\}$
will be foliated by the boundary circles $F(\partial\D,s,t)$
of the holomorphic discs $F(\D,s,t)$, where $|s|<\sqrt{1-t^2}$.

We now want to construct such a holomorphic filling with
the help of the moduli space $\WW_{\varphi}^{\delta}$.
The map described in the next proposition
is defined thanks to Proposition~\ref{prop:ev}.

\begin{prop}
\label{prop:F-delta}
The map
\[ 
F^{\delta}\co (\D\times\D_{1-\delta},\partial\D\times\D_{1-\delta})
\lra (D^4,S^3) \]
defined by
\[ (z,s,t) \longmapsto
\bigl(\ev_1^{-1}\circ\;\varphi\circ u^t_s(1)\bigr)(z)\]
is an embedding.
\end{prop}

Again, the proof of this proposition relies on compactness
and transversality statements and will be given at the end of
Section~\ref{section:transversality}.

For $(s,t)$ in a neighbourhood of $\partial\D_{1-\delta}
\subset\D_{1-\delta}$
we have $F^{\delta}(z,s,t)=u^t_s(z)=F_{\st}(z,s,t)$.
Therefore we obtain a holomorphic filling by setting
\[ F=
\begin{cases}
F^{\delta} & \quad\mbox{\rm on}\quad \D\times\D_{1-\delta},\\
F_{\st}    & \quad\mbox{\rm on}\quad \D\times\{1-\delta\leq s^2+t^2\leq1\}.
\end{cases}
\]
By construction, the maps $F_{\st}$ and $F$ restrict to
diffeomorphisms
\[ f,f_{\st}\co (S^1\times\D_{1-\delta},
S^1\times\partial\D_{1-\delta}) \lra (S^3\setminus\UU^{\delta},
\partial\UU^{\delta}). \]
Observe that
\[ f\circ f_{\st}^{-1}\co S^3\setminus\UU^{\delta}
\lra S^3\setminus\UU^{\delta}\]
is a diffeomorphism equal to the identity near the boundary
of $S^3\setminus\UU^{\delta}$. By slight abuse of notation, we
regard $f\circ f_{\st}^{-1}$ as a diffeomorphism of~$S^3$,
equal to the identity in a neighbourhood of $\overline{\UU}^{\delta}$.

\begin{lem}
\label{lem:isotopy}
The two diffeomorphisms $f\circ f_{\st}^{-1}$ and $\varphi$
of $S^3$ are isotopic relative to a neighbourhood
of~$\overline{\UU}^{\delta}$.
\end{lem}

Before proving this lemma, we show that Cerf's theorem is
now an immediate consequence.

\begin{proof}[Proof of Cerf's theorem $\Gamma_4=0$]
Let $\VV^{\delta}$ be the open subset of $D^4$ defined by
\[ \VV^{\delta}:= F_{\st}(\D\times\{1-\delta<s^2+t^2\leq1\}) .\]
Notice that $\VV^{\delta}\cap S^3=\UU^{\delta}$.
Define $G\co D^4\ra D^4$ by
\[ G:=
\begin{cases}
F\circ F_{\st}^{-1} & \quad\mbox{\rm on}\quad D^4\setminus\VV^{\delta},\\
\id                 & \quad\mbox{\rm on}\quad \VV^{\delta}.
\end{cases}
\]
This is a diffeomorphism of $D^4$ that restricts to
$f\circ f_{\st}^{-1}$ on~$S^3$.

Now an extension of the diffeomorphism $\varphi\co S^3\ra S^3$ to
$D^4$ is defined by sweeping out the isotopy to
$f\circ f_{\st}^{-1}$ over a collar neighbourhood
of $S^3$ in~$D^4$, and then extending over the remaining $4$-ball
as the diffeomorphism~$G$.
\end{proof}

\begin{proof}[Proof of Lemma~\ref{lem:isotopy}]
Consider the diffeomorphism $\chi$ of $S^1\times\D_{1-\delta}$
defined as the composition $\chi = f_{\st}^{-1}\circ\varphi^{-1}\circ f$.
For ease of notation we identify $S^1$ with $\R/2\pi\Z$ and write
$\chi =\chi (\theta ,s,t)$.
This diffeomorphism $\chi$ equals the identity near the boundary.
Our condition on the marked point $1\in S^1=\partial\D$ (corresponding
to $\theta =0$) translates
into saying that $\chi (0,s,t)=(0,s,t)$ for all $(s,t)\in
\D_{1-\delta}$.

Write the components of $\chi$ as $\chi =(\chi^1,\chi^2,\chi^3)$.
The $t$-level Bishop discs for $\varphi$ have boundary in
$\tilde{S}^t=\varphi (S^t)$. This implies $\chi^3(\theta ,s,t)=t$.

The map $f_{\st}$ sends the $s$-curves to the leaves of the
characteristic foliation $S^t_{\xi}$, hence $\varphi\circ f_{\st}$
sends those curves to the leaves of~$\tilde{S}^t_{\xi}$.
On the other hand, the $\theta$-curves are mapped by $f$ to curves
positively transverse to the leaves of~$\tilde{S}^t_{\xi}$, thanks to the
maximum principle. We conclude $\partial_{\theta}\chi^1>0$.
This implies that, for each fixed~$t$, the images of the $\theta$-curves
under $\chi$ are graphs of functions $s=s(\theta )$.

The idea mentioned in Section~\ref{section:idea}
of isotoping the boundaries of
the Bishop discs for $\varphi$ to the images under $\varphi$ of the
boundaries of the standard Bishop discs translates into the isotopy
\[ (\theta ,s,t)\longmapsto
(\chi^1,(1-\sigma )\chi^2+\sigma s,t),\;\; \sigma\in [0,1].\]
The condition $\partial_{\theta}\chi^1>0$ and the resulting
fact that $\chi$ maps each $\theta$-curve to a graph guarantee that this
is indeed an isotopy, stationary near the boundary; thanks to this
boundary behaviour
it is enough to observe that for each $\sigma\in [0,1]$ the given map
constitutes an injective immersion.

For each fixed $s,t$ we may regard the function
$\theta\mapsto\chi^1(\theta ,s,t)$ as a strictly increasing function
$[0,2\pi]\ra [0,2\pi]$ sending both $0$ and $2\pi$ to itself.
With this interpretation, an isotopy from the map above for $\sigma =1$
to the identity is given by
\[ (\theta ,s,t)\longmapsto
((1-\tau )\chi^1+\tau\theta ,s,t),\;\; \tau\in [0,1].\] 
Again, this isotopy is stationary near the boundary.

In conclusion, we have found an isotopy from $\chi$ to
the identity, stationary near the boundary, which is equivalent
to having the isotopy claimed in the lemma.
\end{proof}
\section{Compactness}
\label{section:compactness}
In the present section we wish to show that the truncated moduli space
$\WW_{\varphi}^{\delta}$ is compact; in the next section we find
that it is a manifold with boundary by proving 
surjectivity of the relevant linearised Cauchy--Riemann operator.
For both these analytical questions it is convenient and customary
to work not with smooth maps, but
with the space $W^{1,p}$, for some $p>2$, of (equivalence classes of)
$L^p$-functions whose first partial derivatives in the weak sense
exist and are $p$-integrable.

This approach is justified
by elliptic regularity, see~\cite[Appendix~B]{mcsa04}:
(i) a holomorphic curve of class $W^{1,p}$ will
actually be smooth up to the boundary; (ii) a sequence
of holomorphic curves that converges in the $W^{1,p}$-norm
also converges locally uniformly with all its derivatives.

Thus, we now equip the moduli space $\WW_{\varphi}$ with the topology
induced by the $W^{1,p}$-norm, $p>2$, on maps $\D\ra D^4$.
In the Banach space $W^{1,p}(\D ,\R^4)$ we then have tools such as
the implicit function theorem at our disposal.

Let $(u_{\nu})$, with $u_{\nu}$ of level~$t_{\nu}$, be a sequence
in the truncated moduli space $\WW_{\varphi}^{\delta}$ converging to
$u_0\in W^{1,p}(\D ,D^4)$. By passing to a subsequence we may
assume that the sequence of levels $(t_{\nu})$ converges to some
$t_0\in [-\sqrt{1-\delta},\sqrt{1-\delta}]$. By what we just said,
$(u_{\nu})$ converges with all derivatives to~$u_0$. Thus, $u_0$
is again a holomorphic disc. Moreover, the boundary circles
$u_{\nu}(\partial\D )\subset \tilde{S}^{t_{\nu}}\setminus
\{ q^{t_{\nu}}_{\pm}\}$ are homotopically non-trivial and stay outside the
neighbourhood $\UU^{\delta}$ of the poles $q^t_{\pm}$. So the same will be
true for~$u_0$. In other words, $\WW_{\varphi}^{\delta}$
is a closed subset of $W^{1,p}(\D ,D^4)$.

Our proof of the following proposition uses
methods from~\cite{hofe99}, cf.\ in particular
the proof of Proposition~3.15 in that paper.

\begin{prop}
\label{prop:compact}
The truncated moduli space $\WW_{\varphi}^{\delta}$ is compact.
\end{prop}

\begin{proof}
Let $(u_{\nu})$ be a sequence in $\WW_{\varphi}^{\delta}$, where
$u_{\nu}$ is of level~$t_{\nu}$.
After passing to a subsequence we may assume that $t_{\nu}\ra t_0\in
[-\sqrt{1-\delta},\sqrt{1-\delta}]$.
By \cite[Theorem~B.4.2]{mcsa04}, in order to prove compactness
(i.e.\ to find a converging subsequence with respect to the
$W^{1,p}$-norm), we need to establish a uniform $L^p$-bound for
the sequence $(|\nabla u_{\nu}|)$. We claim that the sequence
$(|\nabla u_{\nu}|)$ is uniformly bounded even in the supremum norm
on the closed disc~$\D$.
  
Arguing by contradiction, assume that such a uniform bound does
not exist. We can then find a sequence of points
$z_{\nu}\ra z_0$ in $\D$ such that $|\nabla u_{\nu}(z_{\nu})| \ra\infty$.
The classical convergence theorems of Montel and Weierstra{\ss}
preclude this at interior points: the maximum principle
provides us with a $C^0$-bound on $(u_{\nu})$ needed for Montel's theorem,
which then guarantees the existence of a locally uniformly convergent
subsequence on $\Int\D$; the theorem of Weierstra{\ss}
tells us that the limit function $u$ is holomorphic and
$\nabla u_{\nu}(z_{\nu})\ra\nabla u(z_0)$, which means that
this sequence of gradients remains bounded.

So necessarily $z_0\in\partial\D$.
Choose a conformal map from the closed upper half-plane
$\Hp\subset\C$ to $\D\setminus\{ -z_0\}$ that sends $0$ to~$z_0$.
(This extends to a conformal map from $\Hp\cup\{\infty\}
\subset\hat{\C}:=\C\cup\{\infty\}$ to~$\D$.)
The differential of this map is bounded from above and below near~$0\in\Hp$,
so by precomposing with this conformal map we may regard the~$u_{\nu}$,
by slight abuse of notation, as maps
\[ u_{\nu}\co (\Hp ,\R )\lra (D^4,\tilde{S}^{t_{\nu}}\setminus
\{ q^{t_{\nu}}_{\pm}\} ),\]
and the sequence $(z_{\nu})$ as a sequence in $\Hp$ converging to~$0$,
still satisfying
\[ R_{\nu}:=|\nabla u_{\nu} (z_{\nu})| \ra\infty.\]
Notice that by Proposition~\ref{prop:energy}
and the remark following it we have a uniform bound
on the Dirichlet energy, i.e.\ there is a constant $C$ such that
\[ \frac{1}{2}\int_{\Hp}|\nabla u_{\nu}|^2\leq C
\;\;\mbox{\rm for all}\;\;\nu\in\N.\]

Choose a sequence $\varepsilon_{\nu}\searrow 0$ such that $\varepsilon_{\nu}
R_{\nu}\ra\infty$, e.g.\ $\varepsilon_{\nu}=1/\ln R_{\nu}$.
Hofer's Lemma~\cite[Lemma~6.4.5]{hoze94},
applied to the continuous function $z\mapsto|\nabla u_{\nu}(z)|$,
allows us to modify the sequences $(z_{\nu})$ and $(\varepsilon_{\nu})$
in such a way that, in addition to the previous conditions, we have
the uniform estimates
\[ |\nabla u_{\nu}(z)| \leq 2R_{\nu}\;\;\mbox{\rm for all}\;\;
z\in\Hp \;\;\mbox{\rm with}\;\; |z-z_{\nu}|\leq\varepsilon_{\nu}.\]
After passing to a further subsequence, we may assume that
\[ R_{\nu}\dist(z_{\nu},\partial\Hp) \ra r\;\;
\mbox{\rm for some}\;\; r\in [0,\infty].\]
When we write $z_{\nu}=x_{\nu}+iy_{\nu}$, this reads as
$R_{\nu}y_{\nu}\ra r$.
We shall deal separately with two cases: either $r<\infty$ or $r=\infty$.

\vspace{1mm}

{\em First case:\/} $r<\infty$.
(i) In a first step we are going to show
that a rescaled subsequence of $(u_{\nu})$ converges
to a holomorphic disc. In a second step we then prove that this
limit disc has contradictory properties.
Replace the sequence $(u_{\nu})$ by the
rescaled sequence $(w_{\nu})$, defined by
\[ w_{\nu}(z):= u_{\nu}(x_{\nu}+z/R_{\nu}).\]
Since this amounts to a conformal change of parameters, the Dirichlet energy
of the $w_{\nu}$ is still bounded by~$C$. The uniform gradient estimate
now becomes, with $\zeta_{\nu}:= iR_{\nu}y_{\nu}$,
\[ |\nabla w_{\nu}(z)| \leq 2 \;\; \mbox{\rm for all}\;\;
z\in\Hp\;\;\mbox{\rm with}\;\;
|z-\zeta_{\nu}|\leq\varepsilon_{\nu}R_{\nu}.\]
Notice that $\zeta_{\nu}\ra ir$ and
$|\nabla w_{\nu}(\zeta_{\nu})|=1$.

We wish to apply \cite[Theorem~B.4.2]{mcsa04} in order
to extract a $C^{\infty}_{\loc}$-converging subsequence of $(w_{\nu})$.
That theorem allows us to have varying almost complex structures,
but requires a fixed boundary condition. Therefore we modify our set-up
as follows. Choose a sequence $(\Psi_{\nu})$ of diffeomorphisms
of $\C^2$ with the following properties:
\begin{itemize}
\item[(1)] $\Psi_{\nu}$ equals the identity outside a ball of radius~$2$,
\item[(2)] $\Psi_{\nu}\ra\id$ in the $C^{\infty}$-topology,
\item[(3)] $\Psi_{\nu}(\tilde{S}^{t_{\nu}})=\tilde{S}^{t_0}$ and
$\Psi_{\nu}(q^{t_{\nu}}_{\pm})=q^{t_0}_{\pm}$.
\end{itemize}
Such $\Psi_{\nu}$ can be constructed by suitably cutting off
the gradient flow of the function $H\circ\varphi^{-1}$ on~$S^3$,
where $H$ is the height function defined in Section~\ref{section:idea}.

Set $\hat{w}_{\nu}:=\Psi_{\nu}\circ w_{\nu}$ and
$J_{\nu}:= T\Psi_{\nu}\circ J_0\circ T\Psi_{\nu}^{-1}$,
so that $J_{\nu}\ra J_0$.
Thus, we now have a sequence $(\hat{w}_{\nu})$
of $J_{\nu}$-holomorphic maps
\[ \hat{w}_{\nu}\co (\Hp,\R)\lra (D^4,
\tilde{S}^{t_0}\setminus\{q^{t_0}_{\pm}\}).\]
The $\hat{w}_{\nu}$ are still subject to the same estimates
as the $w_{\nu}$, possibly after replacing the relevant constants
by larger ones depending only on the sequence~$(\Psi_{\nu})$.

Now \cite[Theorem~B.4.2]{mcsa04} does indeed permit us to
select a subsequence of $(\hat{w}_{\nu})$
converging in $C^{\infty}_{\loc}$ to a non-constant $J_0$-holomorphic map
\[ w\co (\Hp,\R)\lra (D^4,\tilde{S}^{t_0}\setminus\{q^{t_0}_{\pm}\}).\]
The cited theorem requires the boundary condition to be given
by a closed totally real submanifold. No problems arise from our working
with the punctured $2$-sphere $\tilde{S}^{t_0}\setminus\{q^{t_0}_{\pm}\}$,
since we know a priori that
$w_{\nu}$ sends the boundary circle $\partial\Hp\cup\{\infty\}$
to $\tilde{S}^{t_{\nu}}\setminus\UU^{\delta}$.

The $C^{\infty}_{\loc}$-convergence ensures that the energy of
the limit $w$ is estimated from above by~$C$. By removal of singularities,
see \cite[Appendix~B]{lazz00} and~\cite[Theorem~4.7.3]{siko94},
$w$ extends to a holomorphic
map defined on $\Hp\cup\{\infty\}$. By reversing our change of conformal
parameters, we regard this again as a map
\[ w\co (\D ,\partial\D)\lra (D^4,
\tilde{S}^{t_0}\setminus\{q^{t_0}_{\pm}\}),\]
i.e.\ we have an honest holomorphic disc.

\vspace{1mm}

(ii) We now want to show that such a holomorphic disc $w$ cannot possibly
exist. As we saw in the proof of Proposition~\ref{prop:simple},
the map $u|_{\partial\D}$ is an embedding transverse to
the characteristic foliation $\tilde{S}^{t_0}_{\xi}$.

The leaf space of the characteristic foliation $\tilde{S}^t_{\xi}$,
for each $t\in (-1,1)$ and with the singular points $q^t_{\pm}$
removed, can be identified with $S^1$, where $\theta\in S^1$
corresponds to the leaf $s\mapsto \varphi\circ F_{\st}(e^{i\theta},s,t)$.
Let $\sigma$ be the homogeneous measure on $S^1$ of total
measure~$1$, say.

We now regard $u|_{\partial\D}$
as a map from $\partial\D$ to the leaf space $S^1$
of $\tilde{S}^{t_0}_{\xi}$; the maps $u_{\nu}|_{\partial\Hp}$
and $w_{\nu}|_{\partial\Hp}$ will be viewed similarly.
Choose a segment $I\subset S^1\setminus\{-z_0\}$ containing~$z_0$
such that $\sigma (u(I))\geq 15/16$. Under the identification
of $\partial\D\setminus\{ -z_0\}$ with $\partial\Hp$ we take
$I$ of the form $I=[-R,R]$.

Because of the $C^{\infty}_{\loc}$-convergence we have
$\sigma (w_{\nu}(I))\geq 7/8$ for $\nu$ sufficiently large.
With $I_{\nu}:= [x_{\nu}-R/R_{\nu},x_{\nu}+R/R_{\nu}]$ this means that
$\sigma (u_{\nu}(I_{\nu}))\geq 7/8$. The length of the interval
$I_{\nu}$ tends to zero, so at least one of the three intervals in
$\partial\Hp\cup\{\infty\}=\partial\D$ between the three
marked points $1, i, -1$ will be disjoint from~$I_{\nu}$ for
$\nu$ sufficiently large; call this interval~$I'$. The condition
on marked points of our holomorphic discs implies that $\sigma (u_{\nu}(I'))
=\sigma (I')$, which equals $1/4$ or $1/2$. Hence $\sigma (u_{\nu}(I))+
\sigma (u_{\nu}(I'))>1$, contradicting the fact that
$u_{\nu}(I)\cap u_{\nu}(I')=\emptyset$.
 
This contradiction completes the proof in the first case.

\vspace{1mm}

{\em Second case:\/} $r=\infty$.
In this case we define the rescaled sequence $(w_{\nu})$ by
\[ w_{\nu}(z):= u_{\nu}(z_{\nu}+z/R_{\nu})
\;\;\mbox{\rm for}\;\; z=x+iy\;\; \mbox{\rm with}\;\;
y\geq -y_{\nu}R_{\nu}.\]
Then $|\nabla w_{\nu}(0)|=1$, the Dirichlet energy of the $w_{\nu}$
is bounded by~$C$, and we have
the uniform estimate
\[ |\nabla w_{\nu}(z)| \leq 2 \;\; \mbox{\rm for all}\;\;
z\in\Hp\;\;\mbox{\rm with}\;\;
|z-z_{\nu}|\leq\varepsilon_{\nu}R_{\nu}.\]

Again we may appeal to~\cite[Theorem~B.4.2]{mcsa04};
this now gives us a subsequence of $(w_{\nu})$
that converges in $C^{\infty}_{\loc}$ to
a holomorphic map $w\co\C\ra\C^2$, which is bounded by~$1$
and non-constant because of $|\nabla w(0)|=1$.
By the classical {\it Hebbarkeitssatz\/} of Riemann, now applied to a
conformal chart of $\hat{\C}=\C\cup\{\infty\}$ near $\infty$,
we obtain a non-constant holomorphic sphere $w:S^2\ra\C^2$,
in contradiction to the maximum principle.
\end{proof}

\begin{rem}
With a view towards potential generalisations of this
result it is opportune to remark that the references to
classical theorems of complex analysis can be substituted
by results that hold for almost complex structures
tamed by a symplectic form.

For proving that $(|\nabla u_{\nu}|)$ is uniformly bounded one
may replace the reference to the convergence theorems of
Montel and Weierstra{\ss} by arguments from~\cite[Section~4.2]{mcsa04}.
For $z_{\nu}\ra z_0\in\Int\D$ and $|\nabla u_{\nu}(z_{\nu})|\ra
\infty$ these arguments would allow one to infer the
existence of a bubble, i.e.\ a non-constant holomorphic sphere
in~$\C^2$, in violation of the maximum principle.

In place of Riemann's {\em Hebbarkeitssatz\/} one may
cite~\cite[Theorem~4.1.2]{mcsa04}.
\end{rem}
\section{Transversality}
\label{section:transversality}
Here is the main result of the present section:

\begin{prop}
\label{prop:transverse}
The truncated moduli space of Bishop discs $\WW_{\varphi}^{\delta}$
is a $2$-dimensio\-nal manifold with boundary
$\partial\WW_{\varphi}^{\delta}=\{u_s^t\co s^2+t^2=1-\delta\}$.
\end{prop}

For the proof of this proposition we shall establish automatic transversality
results as in \cite{grom85} and~\cite{hls97}; the attribute `automatic' refers
to the fact that no perturbation of the (almost) complex structure
is required to guarantee surjectivity of the Fredholm operator
in question.

We initially drop the condition on the three marked points.
Thus, write $\widetilde{\WW}_{\varphi}$ for the {\bf free moduli space}
of all $t$-level Bishop discs for~$\varphi$, where $t$
varies in the interval $(-1,1)$,
but now without any restriction on the image of the points $1,i,-1$.
Our aim will be to show that $\widetilde{\WW}_{\varphi}$
is a $5$-dimensional manifold. Once this has been achieved,
the following simple lemma allows us to
deduce that our original moduli space $\WW_{\varphi}$ is a
$2$-dimensional manifold.

The tangent space $T_uW^{1,p}(\D,\R^4)$ to the Banach space
$W^{1,p}(\D,\R^4)$ at some point
$u$ of that Banach space can of course be naturally identified with
the Banach space itself. When the target space is not a linear space,
but some manifold $W$, the tangent space $T_uW^{1,p}(\D,W)$
to the Banach manifold $W^{1,p}(\D ,W)$
is the space of $W^{1,p}$-sections of the pullback bundle
$u^*TW$, cf.~\cite[Chapter~3]{mcsa04} or
\cite[Section~6.2]{abklr94}; in other words,
a tangent vector $\eta$ at $u$ is a map $\eta\co\D\ra T\R^4$ with
$\eta (z)\in T_{u(z)}\R^4$. We also want to consider Banach manifolds
of relative maps $(\D ,\partial\D)\ra (\R^4,\Sigma )$, where
$\Sigma\subset S^3$ is a totally real submanifold. Then tangent vectors
are sections of the bundle pair (cf.\ Section~\ref{section:idea})
$u^*(T\R^4,T\Sigma )$. The details of this relative case are worked out
in~\cite[Section~3.1]{zehm03}. The key point here is to
choose an auxiliary metric for which $\Sigma$ is totally
geodesic in~$\R^4$. Then the composition of a section of
$u^*(T\R^4,T\Sigma )$ with the exponential map for this
metric will be a map $(\D ,\partial\D )\ra (\R^4,\Sigma )$.

In particular, if $m_{\tau}$, $\tau\in(-\varepsilon ,
\varepsilon )$, is a holomorphic reparametrisation of~$\D$, with
$m_0=\id_{\D}$, we may regard $\eta :=
(d/d\tau)|_{\tau =0}(u\circ m_{\tau})$
as a tangent vector in this sense; here $\eta (z)$ is actually
tangent to $u(\D )$ at $u(z)$; and for $z\in\partial\D$, tangent
to~$u(\partial\D )$.

See Section~\ref{subsection:Q} for the definition of $Q_k$
used in the following lemma.

\begin{lem}
\label{lem:ev}
The evaluation map
\[
\begin{array}{rccc}
\ev_{1,i,-1}\co & \widetilde{\WW}_{\varphi} & \lra       &
                       S^3\times S^3\times S^3\\
                & u                         & \longmapsto &
                       (u(1),u(i),u(-1))
\end{array}
\]
is transverse to
$\tilde{Q}_0\times\tilde{Q}_1\times\tilde{Q}_2$, where
$\tilde{Q}_k:=\varphi (Q_k)$.
\end{lem}

\begin{proof}
Observe that
$\WW_{\varphi}=\ev_{1,i,-1}^{-1}(\tilde{Q}_0\times
\tilde{Q}_1\times\tilde{Q}_2)$,
so we need to investigate the
differential of the evaluation map at points $u\in\MM_{\varphi}(t)
\subset\WW_{\varphi}$.
This differential sends a tangent vector $\eta$ to the
tangent vector $(\eta (1),\eta (i),\eta (-1))$ at the point
$(u(1),u(i),u(-1))\in S^3\times S^3\times S^3$.

Let $m^k_{\tau}$, $k=0,1,2$, $\tau\in (-\varepsilon ,\varepsilon )$,
be the $1$-parameter family of M\"obius transformations of $\D$
uniquely determined by the condition that the
marked point $i^k$ be sent to $e^{i\tau}i^k$ and the other two
marked points be fixed.

Recall that, by the maximum principle, $u|_{\partial\D}$ 
is an embedding transverse to the characteristic foliation
$\tilde{S}^t_{\xi}$. The leaf $\tilde{\ell}^t_k$ of this characteristic
foliation is given by the transverse intersection of $\tilde{S}^t$
and the disc $\tilde{Q}_k$. This implies that $u|_{\partial\D}$
is transverse to~$\tilde{Q}_k$.

For each $k\in\{ 0,1,2\}$, the tangent vector
$\eta^k:= (d/d\tau)|_{\tau =0}(u\circ m^k_{\tau})$ maps,
under the differential of the evaluation map, to
$(\eta^k(1),\eta^k(i),\eta^k(-1))$. By what we just said, the tangent vector
$\eta^k(i^k)$ is transverse to
$\tilde{Q}_k\subset S^3$,  and $\eta^k(i^l)=0$ for $l\neq k$.
This proves the lemma.
\end{proof}

We now begin with a systematic description of the analytic setting in which
we want to formulate the transversality results that will imply
Proposition~\ref{prop:transverse}. We work with discs of all levels
in $(-1,1)$, although it would suffice to restrict to
a slightly smaller interval, since Bishop discs of level $t$ with
$|t|\geq\sqrt{1-\delta}$ are standard.

By the Sobolev embedding theorem, each equivalence class of maps
of class $W^{1,p}$ has a continuous representative
for $1\cdot p>\dim_{\R}\D =2$. We think of this representative when we
speak simply of a map of class $W^{1,p}$.
This allows us to introduce the following spaces.

\begin{nota}
Let $\CC\subset W^{1,p}(\D ,\R^4)$ be the subset of maps
$u\co (\D ,\partial\D)\ra (\R^4, S^3\setminus K)$ such
that $u|_{\partial\D}\co \partial\D\ra S^3\setminus K\simeq S^1$
lies in the homotopy class of $u^0_0|_{\partial\D}$, where
$u^0_0$ is one of the standard Bishop discs.

Let $\BB\subset\CC$ be the subset of maps $u$ such that
$u|_{\partial\D}$ maps $\partial\D$ into a punctured sphere
$\tilde{S}^t\setminus\{ q^t_{\pm}\}$ for some $t\in (-1,1)$.
\end{nota}

Observe that $\widetilde{\WW}_{\varphi}=
\{ u\in\BB\co\DB u=0\}$. Contrary to appearances,
the equation $\DB u=0$ on $\BB$ is not a linear one,
since the boundary conditions are not linear.

\begin{prop}
The spaces $\BB$ and $\CC$ are Banach manifolds.
\end{prop}

\begin{proof}
We begin with $\CC$. Here we just
give the main points; for more details see~\cite{zehm03}.
First we consider $W^{1,p}$-maps $u\co
(\D ,\partial\D )\ra (\R^4,S^3)$. Choose a metric on $\R^4$
for which $S^3$ is totally geodesic. Then the composition of
a section of $u^*(T\R^4,TS^3)$ with the exponential map for
this metric will be a map $(\D ,\partial\D )\ra (\R^4, S^3)$.
This can be used to show that we obtain a Banach manifold
modelled on the tangent space $T_u\CC$ of $W^{1,p}$-sections
of $u^*(T\R^4,TS^3)$.

Given $u\in\CC$, the set $S^3\setminus K$
is an open neighbourhood of $u(\partial\D )$ in $S^3$. It then follows
from the Sobolev embedding theorem that the space of maps
$(\D ,\partial\D)\ra (\R^4, S^3\setminus K)$
is an open subset of the previous Banach manifold. The space $\CC$,
where we impose an additional homotopical condition, consists
of a connected component of this open subset.

We now turn to $\BB$.
Recall that the $2$-spheres $S^t\subset S^3$ are the level sets
of the function~$H$, given by projection onto the $y_2$-coordinate.
Write $\widetilde{H}:=H\circ\varphi^{-1}$; then the spheres
$\tilde{S}^t=\varphi (S^t)$ are the level sets
of~$\widetilde{H}$. This means that $\BB$ is the inverse
image of the set of constant functions under the smooth map
\[ \begin{array}{rccc}
\Psi\co & \CC & \lra        & C^0(\partial\D , (-1,1))\\
       & u   & \longmapsto & \widetilde{H}\circ u|_{\partial\D}.
\end{array} \]
(Beware that a $W^{k,p}$-function loses derivatives, in general,
when restricted to the boundary, see~\cite[p.~522]{mcsa04}.)

In order to prove that $\BB$ is a submanifold of~$\CC$, we need to
verify that $\Psi$ is transverse to the subspace of constant functions
in $C^0(\partial\D ,(-1,1))$. Write $\R$ for the tangent space to
this subspace. Then the condition to be verified is that
for $u\in\BB$ the composite map
\[ \Psi'\co T_u\CC\stackrel{T_u\Psi}{\lra}
T_{\Psi (u)}C^0(\partial\D ,(-1,1))
=C^0(\partial\D ,\R)\lra
C^0(\partial\D ,\R)/\R \]
is surjective and its kernel splits, cf.~\cite[Proposition~II.2.4]{lang99}.

\vspace{1mm}

(i) We show that $T_u\Psi$ is surjective (for all
$u\in\CC$); this obviously implies surjectivity
of~$\Psi'$. Choose an auxiliary Riemannian
metric $\langle\, .\, ,.\,\rangle$ on~$S^3$. Then the differential
$T_u\Psi$ sends a tangent vector $\eta\in T_u\CC$ to
\[ T_u\Psi (\eta )=
\langle\nabla\widetilde{H}\circ u|_{\partial\D},
\eta|_{\partial\D}\rangle ,\]
which is a real-valued function on~$\partial\D$.

Let $f_0\in C^0(\partial\D ,\R )$ be given.
Poisson's formula for the ball, cf.~\cite{evan98},
gives the unique solution $\eta_0$ of the boundary value problem
\[ \Delta \eta_0 =0,\;\;\; \eta_0 |_{\partial\D}=f_0
\cdot \frac{\nabla\widetilde{H}}{\|\nabla
\widetilde{H}\|^2}\circ u|_{\partial\D},\tag{P}\]
with $\eta_0\co\D\ra\R^4$ continuous on $\D$ and smooth in the
interior of~$\D$. Then $\eta_0\in T_u\CC$ and
$T_u\Psi (\eta_0 )=f_0$.

\vspace{1mm}

(ii) Define $\HH\subset T_u\CC$ as the space of solutions of the boundary
value problem (P) corresponding to functions
$f_0\in C^0(\partial\D,\R)$ with $f_0(1)=0$. We claim that $\HH$ is
complementary to $\ker\Psi'$ in the sense of Banach spaces, i.e.\
$\HH$ is closed in $T_u\CC$, and $\ker\Psi'\oplus\HH =T_u\CC$.

Given $\eta\in T_u\CC$, define $f=T_u\Psi(\eta )\in C^0(\partial\D,\R)$.
Let $\eta_0\in\HH$ be the solution of (P) for
$f_0=f-f(1)$. Then $\eta$ obviously decomposes as $\eta =
(\eta -\eta_0)+\eta_0$, and $\eta-\eta_0\in\ker\Psi'$, since
\[ T_u\Psi (\eta-\eta_0)=f-(f-f(1))=f(1)\]
is a constant function. This shows that $T_u\Psi =\ker\Psi'+\HH$.

If $\eta\in\ker\Psi'\cap\HH$, then $T_u\Psi (\eta )$ is
a constant function on $\partial\D$ which takes the value~$0$
at~$1$, so it is identically zero. By the uniqueness of the
solution of (P) this forces $\eta_0=0$. So we
have a direct sum decomposition of $T_u\Psi$.

It remains to show that $\HH$ is a closed subspace of $T_u\CC$.
Let $(\eta_{\nu})$ be a sequence in $\HH$ converging in
the $W^{1,p}$-norm to some $\eta_0\in T_u\CC$. By elliptic
regularity~\cite[Theorem~B.3.1]{mcsa04},
$(\eta_{\nu})$ is a $W^{k,p}_{\loc}$-Cauchy sequence on the interior
of~$\D$, and by the Sobolev embedding theorem
a $C^k_{\loc}$-Cauchy sequence for all $k\in\N$. It follows
that $\eta_0$ is continuous on $\D$, smooth on the
interior of~$\D$, and it solves the boundary value problem
(P) for a suitable
$f_0\in C^0(\partial\D,\R)$ with $f_0(1)=0$. This means that
$\eta_0\in\HH$, so $\HH$ is closed.
\end{proof}

In order to prove that $\widetilde{\WW}_{\varphi}\subset\BB$
is a $5$-dimensional submanifold, we need to show that
the linearisation $D_u$ of $\DB$ at $u\in\widetilde{\WW}_{\varphi}\subset\BB$
is a surjective Fredholm
operator with a $5$-dimensional kernel. (The splitting of $\ker D_u$
is then given by a closed subspace isomorphic to the
image of~$D_u$.) For the proof
of Proposition~\ref{prop:F-delta} we also have to exhibit
explicit generators of this kernel. Since these are local statements,
we first introduce a suitable local chart around a given
Bishop disc $u\in\widetilde{\WW}_{\varphi}\subset\BB$
of level~$t$. For an open neighbourhood
$U$ of $u(\D)$ in $\C^2$, the set
$\{w\in\BB\co w(\D)\subset U\}$ is, by
the Sobolev embedding theorem, an open neighbourhood of $u$
in~$\BB$. Our aim is to find a neighbourhood $U$
and a chart $\iota\co U\lra\C\times\C$ adapted to our Fredholm
problem.

To this end, choose a frame $e_1,e_2$ of $T\tilde{S}^t$ along
$u(\partial\D )$ with $e_1$ tangent to $u(\partial\D )$,
with $e_2$ tangent to the characteristic foliation
$\tilde{S}^t_{\xi}$, and with $J_0e_2$ --- which is tangent to
$S^3$ but transverse to $\tilde{S}^t$ --- pointing in the
direction of increasing~$\widetilde{H}$. Choose a Riemannian metric
on $\R^4$ which makes $e_1,J_0e_1,e_2,J_0e_2$ an orthonormal
frame along $u(\partial\D )$. Set $e_2'=e_2/\|\nabla\widetilde{H}\|$,
where the gradient and norm are taken relative to the
restriction of the chosen metric to~$S^3$.
Notice that $J_0e_2'$ equals
$\nabla\widetilde{H}/\|\nabla\widetilde{H}\|^2$,
since both vector fields are tangent to~$S^3$,
orthogonal to $\tilde{S}^t$ pointing in the same direction, and have
the same length.
We may extend $u$ to an embedding defined on
a small neighbourhood of $\D$ in~$\C$. Then we find an open neighbourhood
of $\D\times\{ 0\}$ in $\C\times\C$ on which the map
\[ \begin{array}{ccc}
\C\times\C     & \lra        & \C^2 \\
(z_1,x_2+iy_2) & \longmapsto & \exp_{u(z_1)}
                               \bigl( x_2e_2'+y_2J_0e_2'\bigl)
\end{array}\]
defines an embedding. Let $\iota$ be the inverse map,
defined on the image of that embedding. We summarise the properties
of $\iota$, which are obvious from the
construction, in the following lemma.
 
\begin{lem}
\label{lem:chart}
There is a neighbourhood $U$ of $u(\D)\subset\R^4$ and an
embedding $\iota\co U\ra\C\times\C$
with the following properties:
\begin{itemize}
\item[(1)] $\iota (u(z))=(z,0)$ for all $z\in\D$.
\item[(2)] $\iota_*J_0:= T\iota\circ J_0\circ T\iota^{-1}=J_0$
      on $T\R^4|_{\D\times\{0\}}$.
\item[(3)] For each $z\in\partial\D$, the differential $T_{u(z)}\iota$ sends
      the tangent direction along the leaf of the characteristic
      foliation $\tilde{S}^t_{\xi}$ through $u(z)$ to
      $\{ 0\}\times\R\subset T_{(z,0)}\C^2$.
\item[(4)] For each $z\in\partial\D$, the differential $T_{u(z)}\iota$ sends
      $\nabla\widetilde{H}/\|\nabla\widetilde{H}\|^2$ to
      $(0,i)\subset T_{(z,0)}\C^2$.\hfill\qed
\end{itemize}
\end{lem}
  
Thanks to this lemma we may now assume that
$u$ is the inclusion $\D\ra\D\times\{ 0\}\subset\C\times\C$,
at the cost of replacing $J_0$ by a complex structure
$J$ that varies from point to point, but
coincides with $J_0$ along $\D\times\{ 0\}$.

So the linearisation $D_u$ of the Cauchy--Riemann operator
$\CR{J}=\partial_x+J\partial_y$ at the point $u\in
\widetilde{\WW}_{\varphi}\subset\BB$ is given by
\[\begin{array}{rccc}
D_u\co & T_u\BB & \lra        & L^p(\D ,\R^4)\\
       & \eta   & \longmapsto & \DB\eta
                                +(\partial_{\eta}J)\partial_yu.
\end{array}\]
This is a real linear Cauchy--Riemann operator in the sense of
\cite[Appendix~C]{mcsa04} and so is Fredholm.

\begin{prop}
\label{prop:sigma-tau}
The operator $D_u$ is onto and has a $5$-dimensional kernel. This kernel is
spanned by $du(\aut (\D ))$, where $\aut (\D )$ denotes the $3$-dimensional
Lie algebra of infinitesimal holomorphic automorphisms of~$\D$,
and two  smooth sections $\sigma ,\tau$ of $u^*(T\C^2,T\tilde{S}^t)$
and $u^*(T\C^2,TS^3)$, respectively, with
the following properties:
\begin{itemize}
\item[(1)] For $z\in\{ 1,i,-1\}$, the vector
      $\sigma (z)\in T_{u(z)}\tilde{S}^t$,
      is tangent to the characteristic foliation $\tilde{S}^t_{\xi}$.
\item[(2)] The component of the vector field $\tau|_{\partial\D}$ 
      orthogonal to $T\tilde{S}^t$ is a positive constant multiple
      of $\nabla\widetilde{H}/\|\nabla\widetilde{H}\|^2$.
      For $u\in\WW\subset\widetilde{\WW}$ we may assume in addition
      that for $k=0,1,2$ the vector $\tau (i^k)\in T_{u(i^k)}S^3$ is
      tangent to $\tilde{Q}_k$.
\item[(3)] For each $z\in\D$ the vectors $\sigma (z)$ and $\tau (z)$
      span the complement of $du(T_z\D)$ in $T_{u(z)}\R^4$. 
\end{itemize}
\end{prop}

\begin{rems}
(1) The notation $\sigma$ and $\tau$ is meant to be mnemonic.
The component of the vector field $\sigma|_{\partial\D}$
orthogonal to $u(\partial\D )$ points along a leaf of the characteristic
foliation $\tilde{S}^t_{\xi}$ in the direction of the positive $s$-parameter;
the vector field $\tau|_{\partial\D}$ corresponds to an infinitesimal
shift of $t$-level spheres.

(2) Conditions (1) and (2) in the proposition guarantee that
for $u\in\WW$ the tangent vectors $\sigma$ and
$\tau$ lie in the kernel of the differential of the evaluation map
$\ev_{1,i,-1}$, since
\[ T_u\ev_{1,i,-1}(\eta )=(\eta(1),\eta(i),\eta (-1)).\]
So $\sigma$ and $\tau$ may be regarded as
elements of the tangent space $T_u\WW$.
\end{rems}

\begin{proof}[Proof of Proposition~\ref{prop:sigma-tau}]
In our local model around~$u$, the elements $\eta$ of
the tangent space $T_u\BB$ can be written as sections of the form
$(\eta^{\|},\eta^{\perp})\in W^{1,p}(\D ,\C\oplus\C)$.
The condition that $\eta (z)$ be tangent to $S^3$ for
$z\in\partial\D$ translates into
\[ \eta^{\|}(z)\in iz\R \;\;\mbox{\rm for}\;\; z\in\partial\D. \]
The real part of $\eta^{\perp}(z)$ for $z\in\partial\D$ corresponds
to the tangential direction of a leaf of the
characteristic foliation $\tilde{S}^t_{\xi}$;
the imaginary part, to
$\nabla\widetilde{H}/\|\nabla\widetilde{H}\|^2$. The flow of the latter
vector field preserves the level spheres of~$\widetilde{H}$. 
Since the Banach manifold $\BB$ consists of discs whose boundary lies
in such a level sphere, we have the boundary condition
\[ \im\eta^{\perp}(z)=\mbox{\rm const.}\;\;\mbox{\rm for}\;\;
z\in\partial\D. \]

Our real linear Cauchy--Riemann operator $D_u$ now takes the form
\[ \begin{array}{rccc}
D_u\co & T_u\BB                   & \lra        & L^p(\D ,\C\oplus\C )\\
       & (\eta^{\|},\eta^{\perp}) & \longmapsto & (\DB\eta^{\|}+A\eta^{\perp},
                    (\DB+B)\eta^{\perp}),
\end{array} \]
where $A$ and $B$ are smooth maps from $\D$ into
the real endomorphisms of $\C =\R^2$. Elliptic regularity once again
gives us smoothness of the tangent vectors that lie
in $\ker D_u\subset T_u\BB$.

\vspace{1mm}

(i) When the boundary condition on $\eta^{\perp}$ is
strengthened to $\im\eta^{\perp}=0$ for $z\in\partial\D$
--- this corresponds to tangent vectors along the
submanifold of $\BB$ of maps of level~$t$ ---, the vector
$(\eta^{\|}(z),\eta^{\perp}(z))$ lies in the totally
real subspace $iz\R\oplus\R\subset\C\oplus\C$ for each $z\in\partial\D$.
This means we are dealing with a Riemann--Roch boundary value
problem in the sense of~\cite[p.~544]{mcsa04}. The Maslov index
of the bundle pair describing the boundary condition equals~$2$,
so by \cite[Theorem~C.1.10]{mcsa04} the operator $D_u$ is
a surjective Fredholm operator of index~$4$. {\em A fortiori},
$D_u$ is onto when the boundary condition is relaxed to
$\im\eta^{\perp}(z)=\mbox{\rm const.}$ for $z\in\partial\D$.

\vspace{1mm}

(ii) We first consider tangent vectors of the form $(\eta^{\|},0)$.
Such a tangent vector lies in $\ker D_u$ precisely
when $\eta^{\|}$ is holomorphic.
By expanding $\eta^{\|}$ into a power series
on~$\D$, it is easy to see that the condition $\eta^{\|}(z)\in iz\R$
for $z\in\partial\D$ is equivalent to $\eta^{\|}$ being of the
form
\[ \eta^{\|}(z)=a+ibz-\bar{a}z^2\]
for some $a\in\C$ and $b\in\R$, cf.~\cite[p.~75]{hofe99}.
It is a nice calculation to verify that these are indeed
the infinitesimal M\"obius transformations of~$\D$. In other words,
the subspace of $\ker D_u$ of elements of the form $(\eta^{\|},0)$
is the $3$-dimensional space
$\aut(\D)\times\{0\}$.

\vspace{1mm}

(iii) If $(\eta^{\|},\eta^{\perp})$ is an element of $\ker D_u$,
then in particular $(\DB+B)\eta^{\perp}=0$, and for all $z\in
\partial\D$ we have $\eta^{\perp}(z)=ci$ for some real constant~$c$;
we call such an $\eta^{\perp}$ {\bf admissible}.
Any admissible $\eta^{\perp}$ can be written as $\eta^{\perp}=w+ci$, where
$w$ is a smooth solution of the Riemann problem
\[ (\DB+B)w=-cBi,\;\;\; w(z)\in\R\;\;\mbox{\rm for}\;\; z\in\partial\D.
\tag{$\mathrm{R}_c$}\]
(Beware that on the right-hand side of this equation,
$B$ is a real endomorphism acting on the vector $i\in\C=\R^2$.)
Conversely, any solution $w$ of $(\mathrm{R}_c)$ gives rise to an admissible
$\eta^{\perp}=w+ci$.

The boundary condition in $(\mathrm{R}_c)$ defines a bundle pair of
Maslov index~$0$,
so by the Riemann--Roch theorem~\cite[Theorem~C.1.10]{mcsa04}
the operator $\DB +B$ is a surjective Fredholm operator of index~$1$.
Choose a non-trivial solution $w_0$ of $(\mathrm{R}_0)$ and any solution
$w_1$ of~$(\mathrm{R}_1)$. If $w_c$ is a solution of $(\mathrm{R}_c)$, then
$w_c-cw_1$ is a solution of~$(\mathrm{R}_0)$. It follows that the space of
pairs $(w_c,c)$ with $w_c$ a solution of $(\mathrm{R}_c)$ is $2$-dimensional,
spanned by $(w_0,0)$ and $(w_1,1)$.
We conclude that the space of admissible $\eta^{\perp}$ is spanned
by $\eta_0^{\perp}:=w_0$ and $\eta_1^{\perp}:=w_1+i$.

\vspace{1mm}

(iv) Next we want to show that any admissible $\eta^{\perp}$
is in fact the component of a tangent vector $(\eta^{\|},\eta^{\perp})$
that lies in $\ker D_u$. Thus, given an admissible~$\eta^{\perp}$,
we are now looking for a solution $\eta^{\|}\co\D\ra\C$
of the equation $\DB\eta^{\|}=-A\eta^{\perp}$, with $\eta^{\|}(z)\in iz\R$
for all $z\in\partial\D$. This boundary condition defines
a bundle pair of Maslov index~$2$, so again by the Riemann--Roch
theorem~\cite[Theorem~C.1.10]{mcsa04} the operator $\DB$
is a surjective Fredholm operator of index~$3$. By (ii) we find
a unique solution $\eta^{\|}$ with prescribed values in
$iz\R$ at the points $z=1,i,-1$.

Starting with $\eta_0^{\perp}$ or $\eta_1^{\perp}$ from (iii)
we obtain the tangent vectors $\sigma$ and $\tau$ that satisfy
conditions (1) and (2), respectively, of the proposition.
The $\eta^{\|}$-component of $\sigma$ has to vanish at
the points $1,i,-1$, since the real part of $\eta^{\perp}$
corresponds to the direction of the characteristic foliation
$\tilde{S}^t_{\xi}$ along $\partial\D$; for $u\in\WW$ the
$\eta^{\|}$-component $\eta_1^{\|}$ of $\tau$ is determined by the condition
\[ \eta_1^{\|}(i^k)+\eta_1^{\perp}(i^k)\in T_{u(i^k)}\tilde{Q}_k,\;\;
k=0,1,2.\]

\vspace{1mm}

(v) It remains to verify condition~(3). For this we need to show that
an $\R$-linear combination $\lambda_0w_0+\lambda_1(w_1+i)$ that vanishes
at some point in $\D$ has to be trivial.
The $1$-dimensional Carleman similarity
principle~\cite[Corollary~13]{hofe93} guarantees that any solution
$\eta^{\perp}$ of the equation $(\DB+B)\eta^{\perp}=0$
can be written in the form $\eta^{\perp}=e^gf$,
where $f,g\co\D\ra\C$
are continuous functions with $f$ holomorphic in the interior of~$\D$
and $g(\partial\D )\subset\R$. Now write $\lambda_0w_0+\lambda_1(w_1+i)$
in this form. We need to show that $\lambda_0=\lambda_1=0$
if $f$ has any zeros.

Suppose there is a point $z_0\in\D$
with $f(z_0)=0$. If $z_0\in\partial\D$, then $\lambda_1=0$.
If $z_0$ lies in the interior of~$\D$, we may assume after a
M\"obius transformation of $\D$ that $z_0=0$. The function~$f$,
being continuous on $\D$ and holomorphic on $\Int\D$,
satisfies the mean value property
\[ 0=f(0)=\frac{1}{2\pi}\int_0^{2\pi}f(e^{i\theta})\, d\theta.\]
By considering the imaginary part of this equation we see
once again that $\lambda_1=0$.

This means that the function $\im f$, which is harmonic on the
interior of $\D$ and continuous on~$\D$, is zero on $\partial\D$,
and hence zero on $\D$ by the maximum principle. So the holomorphic
function $f$ is constant, and since it vanishes in~$z_0$
it is the zero function, which implies $\lambda_0=0$.
\end{proof}

\begin{proof}[Proof of Proposition~\ref{prop:ev}]
The moduli space $\WW_{\varphi}^{\delta}$ has been
established as a smooth compact surface.
So the image of the evaluation map $\ev_1$ will likewise be compact.
The differential $T_u\ev_1$ has full rank by the preceding
proposition, since $T_u\ev_1(\eta )=\eta(1)$.
So the image of $\ev_1$ will also be open in $\tilde{Q}^{\delta}$,
which means that $\ev_1$ is in fact surjective. Injectivity of $\ev_1$
follows from Proposition~\ref{prop:embedded}.
\end{proof}

\begin{proof}[Proof of Proposition~\ref{prop:F-delta}]
Since the bidisc $\D\times\D_{1-\delta}$ is a compact manifold, we need
only show that $F^{\delta}$ is an injective immersion. Injectivity follows
from all Bishop discs being embedded and mutually disjoint. That
the differential of $F^{\delta}$ has full rank is an immediate
consequence of Proposition~\ref{prop:sigma-tau}~(3).
\end{proof}
\section{Topological intersection of discs}
\label{section:intersection}
In Section~\ref{section:positivity} below we want to establish
results about the positivity of intersections of holomorphic
discs that are instrumental for proving Proposition~\ref{prop:embedded}.
The present section deals with the topological preliminaries.
The main results in this and the following section can be
found in~\cite{ye98}, but we feel that a more self-contained
presentation is appropriate in our context. Some subtle
issues concerning intersection points accumulating at the boundary
(see Proposition~\ref{prop:distinct-nicely} below)
have not been addressed in the cited paper.
Intersection properties of holomorphic discs are also
discussed in~\cite[Section~4.4]{wend05}.

Let $W$ be an oriented $4$-dimensional manifold with
boundary $M=\partial W$, and $\Sigma\subset M$ an embedded 
oriented surface.
Our aim in this section is to define an intersection pairing
on the relative homotopy group $\pi_2(W,\Sigma )$. Elements
in this group are based homotopy classes of $C^0$-maps
$(\D ,\partial\D )\ra (W,\Sigma )$.

\begin{defn}
Let $\DD$ be the space of $C^0$-maps $(\D ,\partial\D )\ra
(W,\Sigma )$. A disc $u\in\DD$ is called {\bf admissible} if
\begin{itemize}
\item[(1)] $u$ is smooth,
\item[(2)] $u$ maps the interior of $\D$ into the interior of~$W$,
\item[(3)] $u|_{\partial\D}$ is an immersion, and
\item[(4)] $u$ is transverse to~$M$.
\end{itemize}
Write $\AAA\subset\DD$ for the subset of admissible discs.
\end{defn}

Observe that in the case of $W=D^4$, $M=S^3$, and $\Sigma\subset S^3$
a totally real surface, any {\em holomorphic} disc $u\in\DD$ is admissible
thanks to regularity and the maximum principle.

The following is a consequence of standard approximation and general
position results in differential topology,
see~\cite[Chapter~2]{hirs76}.

\begin{prop}
\label{prop:admissible}
The subset of admissible discs $\AAA$ is dense in~$\DD$.\hfill\qed
\end{prop}

The set of all {\bf intersection points} of two discs $u_1,u_2\in\DD$ is
\[ S(u_1,u_2):=\{ (z_1,z_2)\in\D\times\D\co u_1(z_1)=u_2(z_2)\}.\]
For the purpose of defining an intersection number of $u_1$ and $u_2$
we distinguish the following subsets of $S(u_1,u_2)$.

\begin{defn}
The set of {\bf interior intersection points} of $u_1,u_2\in\DD$ is
\[ S_{\Int}(u_1,u_2):=S(u_1,u_2)\cap \bigl( \Int\D\times\Int\D\bigr) .\]
The set of {\bf boundary intersection points} is
\[ S_{\partial}(u_1,u_2):=S(u_1,u_2)\cap (\partial\D\times\partial\D ).\]
\end{defn}

By the requirement that an admissible disc map $\Int\D$ to the
interior of~$W$, there can be no mixed intersection points, hence
\[ S(u_1,u_2)=S_{\Int}(u_1,u_2)\sqcup S_{\partial}(u_1,u_2).\]

\begin{defn}
We say that two admissible discs $u_1,u_2\in\AAA$ {\bf intersect nicely}
if $S(u_1,u_2)$ is a finite set of isolated intersection points, and
if for any $(z_1,z_2)\in S_{\partial}(u_1,u_2)$ the two $2$-planes
\[ du_1(T_{z_1}\D),\; du_2(T_{z_2}\D)\subset T_{u_1(z_1)}W
=T_{u_2(z_2)}W\]
either are transverse or coincide. Write $\PP\subset\AAA\times\AAA$
for the set of pairs of nicely intersecting discs.
\end{defn}

Notice that the second condition of this definition is automatically
satisfied for any two holomorphic discs.

\begin{prop}
\label{prop:nice}
The set $\PP$ of pairs of nicely intersecting discs
is dense in $\DD\times\DD$.
\end{prop}

\begin{proof}
By Proposition~\ref{prop:admissible} it suffices to show that
$\PP$ is dense in~$\AAA\times\AAA$. Let $u_1,u_2\in\AAA$ be given.
We are going to find an arbitrarily $C^{\infty}$-small
isotopic perturbation of $u_1$ to a new disc intersecting
$u_2$ nicely. The necessary differential topological
methods are supplied by~\cite{hirs76}.

In a first step, we find an
arbitrarily $C^{\infty}$-small homotopy of $u_1|_{\partial\D}$
to an immersion transverse to $u_2|_{\partial\D}$. By sweeping
out this homotopy over a collar neighbourhood of $M$ in~$W$,
this extends to a homotopy of~$u_1$. We continue to write
$u_1$ for the new map after this and each of the following
perturbations.

So after this first step the maps $u_1|_{\partial\D}$ and
$u_2|_{\partial\D}$ are transverse immersions of the circle
into~$\Sigma$, but $u_1$ and $u_2$ may fail to be transverse
along~$\partial\D$. This, however, can only happen at
one of the finitely many intersection points of $u_1(\partial\D)$
and $u_2(\partial\D)$. In a collar neighbourhood $(-\varepsilon ,0]
\times M$ of $M$ in $W$ one can define a 
$C^{\infty}$-small perturbation of~$u_1$,
supported in this collar neighbourhood and relative to~$\partial\D$,
in a direction transverse to $\Sigma\subset M$.

After this second perturbation we may assume that $u_1$ and $u_2$
are transverse near~$\partial\D$. Then $S_{\Int}(u_1,u_2)$
does not have any accumulation points at the boundary, and a further
perturbation of $u_1$ outside a neighbourhood of $\partial\D$ in $\D$
gives us a nicely intersecting pair $(u_1,u_2)$.
\end{proof}

\begin{rem}
The proof shows that even the subset of pairs $(u_1,u_2)$ with
transverse intersections only is dense in $\AAA\times\AAA$.
\end{rem}

As we shall see in the next section, two distinct
holomorphic discs will intersect nicely, but not --- a priori ---
transversely. For this reason, the above proposition has been phrased
in terms of nicely intersecting discs. We now wish to define
an intersection number for such discs by assigning an intersection
multiplicity to each isolated intersection point, cf.\
\cite{mcdu94}, \cite[Definition~E.2.1]{mcsa04}.
We distinguish three types of intersections, where in each case we write
$p:=u_1(z_1)=u_2(z_2)$ for $(z_1,z_2)\in S(u_1,u_2)$.

\vspace{1mm}

(i) Interior intersections. Given $(z_1,z_2)\in S_{\Int}(u_1,u_2)$,
we define the intersection multiplicity $(u_1\cdot u_2)_{(z_1,z_2)}$
as in the cited references. Choose neighbourhoods of $p$
in $u_1(\D)$ and $u_2(\D)$ that intersect in $p$ only. Then perturb the
discs inside these neighbourhoods, keeping either neighbourhood
disjoint from the boundary of the other, so as to achieve
transversality of intersection. Then count the intersection points with sign
according to the orientation. Notice that for the definition
of this intersection multiplicity the maps $u_1$ and $u_2$ need not be
smooth; it suffices that the intersection points be isolated.

\vspace{1mm}

(ii) Transverse boundary intersections. Here the intersection
multiplicity can be
defined as for an interior transverse intersection point.
If $(z_1,z_2)\in S_{\partial}(u_1,u_2)$ is a transverse
intersection point, the Whitney sum
\[ du_1(T_{z_1}\D)\oplus du_2(T_{z_2}\D) \]
is isomorphic to $T_pW$.
The intersection multiplicity $(u_1\cdot u_2)_{(z_1,z_2)}$
equals $\pm 1$, depending on whether this isomorphism
is orientation preserving or reversing.

\vspace{1mm}

(iii) Tangential boundary intersections (cf.~\cite{ye98}, \cite{wend05}).
Let $(z_1,z_2)\in S_{\partial}(u_1,u_2)$ be a tangential intersection
point, i.e.\ $du_1(T_{z_1}\D)=du_2(T_{z_2}\D)$.
Since $u_1$ and $u_2$ are transverse to $M$ from the same side,
we can write $u_2 (\D)$ locally as a graph over $u_1(\D)$.
It is not possible, in general, to give a homotopy-invariant
definition of an intersection multiplicity when this `same side' condition
is dropped, see~\cite[Example~4.4.7]{wend05}.
We write the local model for
\[ p\in\Sigma\subset M=\partial W\subset W\]
as
\[ 0\in\R\times\R\subset\R\times\C\subset\Hp\times\C.\]
(This complex model is chosen for convenience of notation only;
we do not consider holomorphic discs yet.) Moreover,
we may assume that the local model has been selected in such a
way that the local description of $u_1$ and $u_2$ is of the form
\[ (\Hp,\R)\lra (\Hp\times\C,\R\times\R)\]
with
\[ u_1(z)=(z,0)\;\;\;\mbox{\rm and}\;\;\;
u_2(z)=(z,h(z)).\]
Now extend $u_1$ and $u_2$ by Schwarz reflection to continuous
maps
\[ (\C,\R)\lra (\C\times\C,\R\times\R).\]
For $u_1$ this simply means
extending the definition $u_1(z)=(z,0)$ to $z\in\C$; for $u_2$ we set
\[ u_2(z)=(z,\overline{h(\oz)})\;\;\mbox{\rm for}\;\;
\im z<0.\]
Up to homotopy, the definition of this extension is independent
of the choice of local coordinates, and so the
intersection multiplicity $(u_1\cdot u_2)_{(z_1,z_2)}$
is well defined as the intersection multiplicity in the
sense of (i) of these extended maps.

\vspace{1mm}

The definition of the intersection multiplicity as in (iii)
can also be used for transverse boundary intersections, in which case
it coincides with that in~(ii). Just as in~(i), this allows us
to assign a multiplicity to any isolated intersection point
of discs that are merely continuous.

\vspace{1mm}

The following example will motivate the definition of the
overall intersection number of a nicely intersecting pair $(u_1,u_2)$.

\begin{ex}
Consider the maps $u_1(z)=(z,0)$ and $u_2(z)=(z,z^3)$, $z\in\Hp$,
in the local model just described. The intersection point $(0,0)$ is
a tangential boundary intersection of multiplicity~$3$. We
can perturb the extended $u_2$ within the class of maps
symmetric with respect to Schwarz reflection to $u_2^{\varepsilon}(z)=
(z,z^3-\varepsilon^3)$, $z\in\C$, with $\varepsilon>0$, say. Then
$u_1$ and $u_2^{\varepsilon}$ intersect in the three points
$\varepsilon \exp (2\pi ik/3)$,
$k=0,1,2$, all of multiplicity~$1$. One of these
intersection points still sits on the original boundary, now as a transverse
intersection point. The two other intersection points correspond
to each other under Schwarz reflection; only one of them lies
inside the actual discs $u_1,u_2^{\varepsilon}$, i.e.\
in the non-extended domain of definition $\Hp$ (in the
local model) of these discs.
\end{ex}

Notice that we have used the condition that $u_1,u_2$ map
$\partial\D$ to $\Sigma\subset M$ in an essential way for the
extension of $u_1,u_2$ via Schwarz reflection. The more obvious
way of extending $u_1,u_2$ by doubling along~$M$, which would
require only that $\partial\D$ map to~$M$,
would lead to intersection points on
the boundary of multiplicity~$0$, and pairs of intersection
points (related by the symmetry) of opposite multiplicity.

\begin{defn}
The {\bf intersection number} of a pair $(u_1,u_2)\in\PP$
of nicely intersecting admissible discs is
\[ u_1\bullet u_2:=2\sum_{S_{\Int}(u_1,u_2)}
(u_1\cdot u_2)_{(z_1,z_2)}+
\sum_{S_{\partial}(u_1,u_2)}
(u_1\cdot u_2)_{(z_1,z_2)},\]
where the sums are meant to be taken over the pairs $(z_1,z_2)$
of points in the sets $S_{\Int}(u_1,u_2)$ and $S_{\partial}(u_1,u_2)$,
respectively.
\end{defn}

As observed by Ye~\cite[Lemma~7.2]{ye98}, this intersection
number is a homotopy invariant.

\begin{prop}
\label{prop:intersection}
The intersection number $u_1\bullet u_2$ depends only on the
homotopy classes $[u_1],[u_2]\in\pi_2(W,\Sigma )$.
\end{prop}

Our proof follows a route slightly different from Ye's. Before we turn
to the formal argument, we illustrate by an example the kind of
deformations we shall employ.

\begin{ex}
We can picture transverse boundary intersections in the same
local model used above to describe tangential boundary intersections.
We always take $u_1(z)=(z,0)$. The two maps
$u_2(z)=(z,z)$ and $u_2^*(z)=(z,-z)$ give
local models for a transverse boundary intersection, both of the
same sign, i.e.\ $(u_1\cdot u_2)_{(0,0)}=(u_1\cdot u_2^*)_{(0,0)}$.
However, the signs of the intersection of the boundary curves
in $\R\times\R$ are different: $(u_1|_{\R}\cdot u_2|_{\R})_{(0,0)}=
-(u_1|_{\R}\cdot u_2^*|_{\R})_{(0,0)}$.

Choose some small $\varepsilon >0$ and let
$h_2\co\R_0^+\ra\R$ be a smooth function with
$h_2(y)=y$ for $y\in [0,\varepsilon )$ and $h_2(y)=-y$ for
$y\geq 2\varepsilon$, with a single further zero at $3\varepsilon/2$,
where $h_2'(3\varepsilon/2)<0$. Then the map $z=x+iy\mapsto
(z,-x+ih_2(y))$ coincides with $u_2^*$ outside a small
neighbourhood of the boundary, but near the
boundary it looks like $z\mapsto (z,-\oz)$.

Suppose for argument's sake that
the orientations in the model had been chosen in such a way
that $(u_1\cdot u_2^*)_{(0,0)}=1$, but $(u_1|_{\R}
\cdot u_2^*|_{\R})_{(0,0)}=-1$.
After the described perturbation near the boundary, both
intersections are negative. The new transverse intersection point in the
interior at $z=3\varepsilon i/2$ is positive, so the intersection
number $u_1\bullet u_2^*$ is not affected by this perturbation.
\end{ex}

This example shows the following. Let $u_1,u_2$ be a
pair of {\em transversely\/} intersecting admissible discs.
Then we may perturb $u_2$ near the boundary
in such a way that the intersection number $u_1\bullet u_2$
remains unchanged, but the intersection multiplicity at each boundary
intersection point equals the intersection multiplicity of
the boundary curves in $\Sigma$ at that point.

Observe that the two models where the intersection multiplicities
coincide in the way described are $z\mapsto (z,z)$ and $z\mapsto
(z,-\oz)$. In either of these models the imaginary
part of the second coordinate equals~$y$, which means that
both models depict discs on the same side of $\Sigma\subset M$
in a collar neighbourhood of $M=\partial W$ in~$W$.

\begin{proof}[Proof of Proposition~\ref{prop:intersection}]
Let $u_1^0$ and $u_1^1$ be admissible discs with $[u_1^0]=[u_1^1]
\in\pi_2(W,\Sigma )$, either of which nicely intersects a
further admissible disc~$u_2$. We need to show that
$u_1^0\bullet u_2=u_1^1\bullet u_2$.

In case there are tangential boundary intersections, we
make a small homotopic perturbation of $u_1^0,u_1^1$
as in the first of
our examples so as to get transverse intersections only.
This does not affect the intersection number. Next, we perform
a homotopy as in the second example to adjust the signs of
the boundary intersections. We continue to write $u_1^0,u_1^1$
for these discs.

Since $u_1^0$ and $u_1^1$ are homotopic, so {\em a fortiori\/}
are the boundary curves $u_1^0|_{\partial\D}$ and
$u_1^1|_{\partial\D}$. Hence
\[ \sum_{S_{\partial}(u_1^0,u_2)} (u_1^0\cdot u_2)_{(z_1,z_2)}=
u_1^0|_{\partial\D}\bullet u_2|_{\partial\D}=
u_1^1|_{\partial\D}\bullet u_2|_{\partial\D}=
\sum_{S_{\partial}(u_1^1,u_2)} (u_1^1\cdot u_2)_{(z_1,z_2)},\]
where $\bullet$ is the usual intersection product of curves
on~$\Sigma$.

A collar neighbourhood of $\Sigma$ in $W$ can be written
globally as $(-\varepsilon ,0]\times\R\times\Sigma$.
Given the observation before
this proof, we may assume that the disc $u_2(\D)$ looks
like $(-\varepsilon ,0]\times\{ 0\}\times u_2(\partial\D)$
in this collar, while the discs $u_1^j$, $j=0,1$,
look like $\{ (-t,t,q)\co t\in [0,\varepsilon ),\, q\in u_1^j(\partial\D)\}$.

Now the disc $u_1^0$, a homotopy from $u_1^0|_{\partial\D}$ to
$u_1^1|_{\partial\D}$, and the disc $u_1^1$ with reversed orientation
define a map $f\co S^2\ra W$. When we push this $2$-sphere
away from the boundary $\partial W$ in the direction of the
tangent vector $(-1,1,0,0)$ to $(-\varepsilon ,0]\times\R\times\Sigma$,
we remove all boundary intersection points without
creating any new intersections. So the intersection number
of this $2$-sphere with $u_2$ equals
\[ \sum_{S_{\Int}(u_1^0,u_2)} (u_1^0\cdot u_2)_{(z_1,z_2)}-
\sum_{S_{\Int}(u_1^1,u_2)} (u_1^1\cdot u_2)_{(z_1,z_2)}.\]
This intersection number is well defined homotopically
(even homologically), i.e.\ there is an intersection product
\[ \bullet\co H_2(W)\otimes H_2(W,\partial W)\lra\Z,\]
and the intersection product $[f]\bullet [u_2]$ is given by the
intersection number above, cf.~\cite[Section~VI.11]{bred93}.
However, from $[u_1^0]=[u_1^1]$ in $\pi_2(W,\Sigma )$ we
deduce that $f$ is homotopically and hence homologically trivial,
so $[f]\bullet [u_2]=0$. This concludes the proof.
\end{proof}

\begin{rem}
Two maps $(\D ,\partial\D )\ra (W,\Sigma )$ that are sufficiently
$C^0$-close are homotopic. Thus, with Proposition~\ref{prop:nice}
we can define an intersection product on $\pi_2(W,\Sigma )$.
\end{rem}
\section{Positivity of intersections of holomorphic discs}
\label{section:positivity}
We now want to apply the topological results of the preceding section to
holomorphic discs $(\D,\partial\D)\ra (D^4,\tilde{S}^t)$. For future
reference we formulate the results in slightly greater generality.
Thus, let $(W,J)$ be a smooth almost complex $4$-manifold with
$J$-convex boundary $\partial W$, and $\Sigma\subset\partial W$
an embedded oriented surface totally real with respect to~$J$.

In the present section, by {\bf holomorphic disc} we shall always mean
a smooth {\em non-constant\/} $J$-holomorphic disc
$(\D,\partial\D)\ra (W,\Sigma )$.
As we saw in the previous section, any
holomorphic disc is admissible.
Two holomorphic discs $u_1,u_2$ are called {\bf distinct}
if $u_1(\D)\neq u_2(\D)$.

\begin{prop}
\label{prop:distinct-nicely}
Any two distinct holomorphic discs intersect nicely.
\end{prop}

\begin{proof}
By the comment after the definition of `nicely intersecting'
in the previous section, it suffices to show that
any two distinct holomorphic discs intersect in finitely many points only.
We are going to prove the contrapositive. That is, let $u_1,u_2$
be two holomorphic
discs for which $S(u_1,u_2)$ is infinite. We have to
show that $u_1(\D)=u_2(\D)$.

In the infinite set $S(u_1,u_2)$ we can choose a non-constant
sequence $(z_1^{\nu},z_2^{\nu})$, $\nu\in\N$, converging to some point
$(z_1^0,z_2^0)\in S(u_1,u_2)$. If $(z_1^0,z_2^0)\in S_{\Int}(u_1,u_2)$,
then by the work of Micallef--White~\cite[Theorem~7.1]{miwh95}
one can find neighbourhoods
$U_i\in\D$ of $z_i^0$, $i=1,2$, with $u_1(U_1)=u_2(U_2)$.
We claim that the same conclusion holds for
$(z_1^0,z_2^0)\in S_{\partial}(u_1,u_2)$.
Accepting this claim for the time being, we then see that no matter
where the intersection points accumulate, the set of points in $\D$ that
have a neighbourhood
mapped by $u_1$ into $u_2(\D)$ is non-empty. This set is
open by definition, and closed by the same argument used to show
that it is non-empty.
Hence $u_1(\D)\subset u_2(\D)$. The converse inclusion
holds by symmetry of the argument.

It remains to prove the claim. As in the topological situation we may
choose local models such that $u_1,u_2$ around $z^0_1,z^0_2$,
respectively, may be regarded as germs of maps
\[ (\Hp,\R,0)\lra (\Hp\times\C,\R\times\R,0)\]
of the form
\[ u_1(z)=(z,0)\;\;\mbox{\rm and}\;\;u_2(z)=(a(z),b(z)).\]
Moreover, we may assume that in this local model the almost complex
structure $J$ coincides with the standard structure $J_0$
along $\Hp\times\{ 0\}\subset\Hp\times\C$. The assumption on
$(z^0_1,z^0_2)$ being an accumulation point of intersections
translates into saying that $b$ has an accumulation point of
zeros in $z=0$.

The following interpolation argument is reminiscent of the proof of the
unique continuation theorem in~\cite[p.~24]{mcsa04}. We write
\[ J(a,b)-J(a,0)  =  \int_0^1\frac{d}{dt}J(a,tb)\, dt
               =  \int_0^1D_2J_{(a,tb)}(b)\, dt.\]
This can now be used to simplify the Cauchy--Riemann equation for~$u_2$.
\begin{eqnarray*}
0\;=\;\partial_xu_2+J(u_2)\partial_yu_2
       & = & \partial_xu_2+J_0\partial_yu_2+(J(u_2)-J_0)\partial_yu_2\\
       & = & \partial_xu_2+J_0\partial_yu_2+(J(a,b)-J(a,0))\partial_yu_2\\
       & = & \partial_xu_2+J_0\partial_yu_2+ \int_0^1D_2J_{(a,tb)}(b)\, dt
             \cdot\partial_yu_2.
\end{eqnarray*}
Define a smooth map $B=(B_1,B_2)$ from $\Hp$ into
the real linear maps $\C\ra\C^2$ by
\[ B_z\eta=\int_0^1D_2J_{(a(z),tb(z))}(\eta)\, dt\cdot\partial_yu_2(z).\]
Then the Cauchy--Riemann equation for $u_2$ decouples into two
$1$-dimensional equations
\begin{align*}
\partial_xa+i\partial_ya+B_1b  &=  0,\tag{$\mathrm{CR}_{a,b}$}\\
\partial_xb+i\partial_yb+B_2b  &=  0.\tag{$\mathrm{CR}_b$}
\end{align*}
Now one applies a relative version of the Carleman similarity
principle to~($\mathrm{CR}_b$).
For this principle in the absolute case see
\cite[Corollary~13]{hofe93} (for the $1$-dimensional situation),
\cite[Section~A.6]{hoze94} and~\cite[Theorem~2.3.5]{mcsa04}.
The key to a relative version of this principle
is the observation (Step~2 in the
proof of~\cite[Theorem~2.3.5]{mcsa04}) that $B_2$ may be assumed to be
complex linear. Then both the solution $b$ of our linear equation
and the map $B_2$ can be extended from $\Hp$ to $\C$ (near $0$)
by Schwarz reflection. The Carleman similarity principle
(cf.\ Step~(v) in the proof of Proposition~\ref{prop:sigma-tau})
then implies that the solution $b$ is a pointwise complex linear
transformation of a holomorphic function. The identity theorem applied to
this holomorphic function yields that this function, and hence~$b$,
is identically zero.
\end{proof}

\begin{rem}
Relative versions of the Carleman similarity principle have been
mentioned previously in \cite[Proposition~3.1]{lazz00}
and~\cite[Lemma~3.1]{abba04}.
\end{rem}

In the following theorem we use $|\, .\, |$ to denote the cardinality of
a finite set.

\begin{thm}[Positivity of intersections]
\label{thm:positivity}
The intersection number of distinct holomorphic
discs $u_1,u_2$ satisfies the inequality
\[ u_1\bullet u_2\geq 2|S_{\Int}(u_1,u_2)|+|S_{\partial}(u_1,u_2)|,\]
with equality if and only if all intersections are transverse.
In particular, $u_1(\D)$ and $u_2(\D)$ are disjoint if and only
if $u_1\bullet u_2=0$.
\end{thm}

\begin{proof}
We need to show that the intersection multiplicities
$(u_1\cdot u_2)_{(z_1,z_2)}$ are greater than or equal to~$1$, with equality
precisely in the case of a transverse intersection. For interior intersection
points this result is due to
Micallef--White~\cite[Theorem~7.1]{miwh95},
cf.~\cite[Proposition~E.2.2]{mcsa04}.

At a transverse boundary intersection point the intersection multiplicity
is~$1$; this is seen exactly as for interior transverse intersections.

For a tangential boundary intersection point we
use the local model from the
preceding proof. Thus, write $u_1$ and $u_2$ near the intersection point
$(z_1,z_2)=(0,0)\in\Hp\times\Hp$ as $u_1(z)=(z,0)$
and $u_2(z)=(a(z),b(z))$, where $a$ and $b$ satisfy the linear
Cauchy--Riemann equations ($\mathrm{CR}_{a,b}$)
and $(\mathrm{CR}_b$), respectively. The condition that the
intersection point be tangential means that $db_0=0$. Then
$da_0\neq 0$, since boundary points are non-singular.
The relative Carleman similarity principle, applied to~($\mathrm{CR}_b$),
tells us that $b$ is of the form $b(z)=b_kz^k+o(|z^k|)$
with $k\geq 2$ and $0\neq b_k\in\C$;
this follows from the observation that the pointwise complex linear
transformation from $b$ to a holomorphic function may be taken
as the identity at $z=0$ by incorporating the transformation at $0$
as a multiplicative constant into the holomorphic function.

Moreover, from ($\mathrm{CR}_a$) we see with $b(0)=0$ that
the differential $da_0$ is complex linear. It follows that
the intersection multiplicity is $k\geq 2$.
\end{proof}

We now want to use these results to investigate the self-intersections
of holomorphic discs. This will yield the criterion for a 
holomorphic disc to be embedded
that we need for the proof of Proposition~\ref{prop:embedded}.

\begin{defn}
Let $A\in\pi_2(W,\Sigma )$ be a relative homotopy class,
represented by a $C^0$-map $u\co (\D ,\partial\D)\ra (W,\Sigma)$.
By Section~\ref{section:intersection}, the intersection number
$A\bullet A$ is well defined.
Write $\mu (A)$ for the Maslov index of the bundle pair
$(u^*TW,(u|_{\partial\D})^*T\Sigma )$ over $(\D,\partial\D)$.
The {\bf embedding defect} of $A$ is
\[ D(A):= A\bullet A-\mu (A)+2.\]
\end{defn}

The set of {\bf self-intersection points} of a 
holomorphic disc $u$ is
\[ S(u):=\{ (z_1,z_2)\in\D\times\D\co u(z_1)=u(z_2),\; z_1\neq z_2\}.\]
Notice that self-intersection points come in pairs $(z_1,z_2)$,
$(z_2,z_1)$.
As in Section~\ref{section:intersection} we can speak of the interior
and the boundary self-intersection points. Write
$\crit (u)\subset\D$ for the set of critical points
of~$u$. Since $u$ is admissible, $\crit (u)$ is contained
in~$\Int\D$ and does not have any accumulation points
on~$\partial\D$. Hence, $\crit (u)$
is a finite set by~\cite[Lemma~2.4.1]{mcsa04},
according to which
critical points cannot accumulate at an interior point.

Before we can prove an estimate on the embedding defect of a 
holomorphic disc
in terms of the disc's self-intersections and critical points,
we need a formula that allows us to compute the Maslov index of
a bundle pair of complex rank~$1$ from the self-intersection of
the zero section, analogous to the formula for
first Chern class of a complex line bundle,
cf.~\cite[Proposition~2.7]{ye98}. We write $\ord_zs$ for the order
of an isolated zero $z$ of a continuous bundle section~$s$, that is,
the intersection index with the zero section.

\begin{lem}
Let $(E,F)$ be a complex rank~$1$ bundle pair over $(\D ,\partial\D)$.
Let $s$ be a continuous section of $(E,F)$ with isolated zeros.
Then the Maslov index of the bundle pair is given by
\[ \mu (E,F)=2\sum_{z\in\Int\D}\ord_zs+\sum_{z\in\partial\D}\ord_zs.\]
\end{lem}

\begin{proof}
We use a doubling argument (based on Schwarz reflection)
as in~\cite[p.~157]{hls97}.
Write $\overline{\D}$ for the disc $\D$ with reversed orientation.
We double $\D$ to the Riemann sphere $S^2=\D\cup_{\partial\D}
\overline{\D}$. Reflection in $\partial\D$ defines an anti-holomorphic
involution on~$S^2$.

The double of $E$ is constructed similarly. Write $\overline{E}$ for
the complex line bundle over $\overline{\D}$ whose fibre
$\overline{E}_{\oz}$ over $\oz\in\overline{\D}$ is $E_z$ with
the conjugate complex structure. The gluing of $E_z$ with
$\overline{E}_{\oz}$ over $z\in\partial\D$ is effected by the anti-complex
involution in~$F_z$.

Then the section $s$ doubles to a continuous section $s\cup\overline{s}$
of the complex line bundle $E\cup\overline{E}$ over~$S^2$,
with the total order of zeros $\ord (s\cup\overline{s})$
given by the right-hand side of the
equation in the lemma. Now the composition formula
for the Maslov index~\cite[Theorem~C.3.5]{mcsa04} and
the relation of the Maslov index with the Chern
class~\cite[Theorem~C.3.6]{mcsa04} yield
\begin{eqnarray*}
2\mu (E,F) & = & \mu (E,F)+\mu (\overline{E},F)
                 \; = \; \mu (E\cup\overline{E},\emptyset)\\
           & = & 2\langle c_1(E\cup\overline{E}),[S^2]\rangle
                 \; = \; 2\ord (s\cup\overline{s}).
\end{eqnarray*}
The claimed formula for the Maslov index follows.
\end{proof}

The following theorem is a quantitative version
of~\cite[Theorem~7.6]{ye98}; cf.~\cite[Theorem~7.3]{miwh95}
for the adjunction inequality in the absolute case.
Recall the definition of a {\em simple\/} disc from
Section~\ref{section:Bishop}.

\begin{thm}[Relative adjunction inequality]
\label{thm:adjunction}
Let $u$ be a simple holomorphic disc. Then the
set $S(u)$ of self-intersections is finite, and the
embedding defect satisfies
\[ D([u])\geq 2|S_{\Int}(u)|+|S_{\partial}(u)|+4|\crit (u)|.\]
For $\crit (u)=\emptyset$ we have equality in this formula
if and only if all self-intersections are transverse.
In particular, $u$ is an embedding if and only if $D([u])=0$.
\end{thm}

\begin{proof}
The finiteness of $S(u)$ is shown as in the proof of
Proposition~\ref{prop:distinct-nicely}; here
the assumption that $u$ be simple is used.

A lemma of Frauenfelder~\cite{frau00}, cf.~\cite[Lemma~4.3.3]{mcsa04}
allows us to choose a Riemannian metric on $W$ for which $J$ is an
orthogonal endomorphism field, and for which the totally real submanifold
$\Sigma$ is totally geodesic.

\vspace{1mm}

{\em First case:\/} $\crit (u)$ is empty. Then $(du(T\D),du(T\partial\D))$
is a subbundle pair of $(u^*TW,(u|_{\partial\D})^*T\Sigma )$ of
Maslov index~$2$. By our choice of metric,
the orthogonal complement of this subbundle pair is again a subbundle
pair; write $\mu^{\perp}$ for its Maslov index. Then
\[ D([u])=u\bullet u-\mu^{\perp}.\]

Choose a smooth section $s$ of this orthogonal bundle transverse to the
zero section. We may assume that the zeros of $s$ are disjoint
from the non-injective points of~$u$, i.e.\ the points $z\in\D$
for which the preimage $u^{-1}(u(z))$ contains a point other than~$z$.
For $s$ sufficiently small, an admissible disc $u_2$ nicely
intersecting $u_1:=u$ can be defined by $u_2:=\exp_{u_1}s$.
By the homotopy invariance
of the intersection number we have $u\bullet u=u_1\bullet u_2$.
The intersection number $u_1\bullet u_2$ is given as a sum over
intersection multiplicities, with interior intersections counting
twice. Diagonal intersection points $(z,z)$ of $u_1$ and $u_2$
correspond exactly to the zeros of the section~$s$. By
the preceding lemma it follows
that the contribution of these diagonal intersection points
to $u_1\bullet u_2$ equals~$\mu^{\perp}$.

We claim that the contribution of the non-diagonal
intersection points to $u_1\bullet u_2$ equals the sum
over the multiplicities of the self-intersection points
of~$u$ (with interior points
counted twice). These non-diagonal intersection points
arise in pairs $(z_1',z_2')$, $(z_1'',z_2'')$ corresponding to
a pair of self-intersection points $(z_1,z_2),
(z_2,z_1)\in S(u)$. So the two sums of intersection
multiplicities are indeed equal. This discussion can be summarised
in the formula
\[ D([u])=2\sum_{S_{\Int}(u)}(u\cdot u)_{(z_1,z_2)}+
\sum_{S_{\partial}(u)}(u\cdot u)_{(z_1,z_2)}.\]
From here the argument concludes as in the proof of
Theorem~\ref{thm:positivity}.

\vspace{1mm}

{\em Second case:\/} $\crit (u)$ is non-empty. Here one can
use a perturbation argument to turn critical points into
self-intersections, see~\cite[Proposition~E.2.4]{mcsa04}.
A critical point $z$
of a holomorphic disc~$u$
is said to be of {\bf order}~$k\in\N$ if at $z$ the $k$-jet of $du$
is the lowest order non-vanishing jet. For instance, a critical
point of order~$1$ is characterised by the non-vanishing of the
Hessian of~$u$.
As shown in~\cite[p.~610]{mcsa04} and~\cite[Theorem~7.3]{miwh95}, 
each critical point of order~$k$ gives rise to at least $k(k+1)$
intersection points.
Since all critical points lie in the interior, these
intersection points are counted twice in the
intersection product, i.e.\ the contribution
to the embedding defect is $2k(k+1)$.

If a critical point happens to be an intersection point
of the original disc, we may think of the situation near
this point as the intersection of two local discs.
When we perturb and desingularise these two local discs,
we obtain transverse intersections of the two local discs
and self-intersections. The former contribute to the
intersection product according to the intersection
index, the latter contribute as analysed in the foregoing paragraph.
\end{proof}


\begin{thebibliography}{10}

\bibitem{abba04}
{\sc C. Abbas},
Pseudoholomorphic strips in symplectisations III.
Embedding properties and compactness,
{\it J. Symplectic Geom.}
{\bf 2} (2004), 219--260.

\bibitem{abklr94}
{\sc B. Aebischer, M. Borer, M. K\"alin, Ch. Leuenberger and
H. M. Reimann},
{\it Symplectic Geometry},
Progr. Math. {\bf 124},
Birkh\"auser Verlag, Basel (1994).

\bibitem{bred93}
{\sc G. E. Bredon},
{\it Topology and Geometry},
Grad. Texts in Math. {\bf 139},
Springer-Verlag, Berlin (1993).

\bibitem{cerf68}
{\sc J. Cerf},
{\it Sur les diff\'eomorphismes de la sph\`ere de dimension trois\/}
$(\Gamma_4=0)$,
Lecture Notes in Math. {\bf 53},
Springer-Verlag, Berlin (1968).

\bibitem{elia90}
{\sc Ya. Eliashberg},
Filling by holomorphic discs and its applications,
in: {\it Geometry of Low-Dimensional Manifolds, Vol.~2\/} (Durham, 1989),
London Math. Soc. Lecture Note Ser. {\bf 151},
Cambridge University Press (1990), 45--67.

\bibitem{elia92}
{\sc Ya. Eliashberg},
Contact $3$-manifolds twenty years since J.~Martinet's work,
{\it Ann. Inst. Fourier (Grenoble)\/}
{\bf 42} (1992), 165--192.

\bibitem{elpo96}
{\sc Ya. Eliashberg and L. Polterovich},
Local Lagrangian $2$-knots are trivial,
{\it Ann. of Math.~(2)\/}
{\bf 144} (1996), 61--76.

\bibitem{evan98}
{\sc L. C. Evans},
{\it Partial Differential Equations},
Grad. Stud. Math. {\bf 19},
American Mathematical Society, Providence, RI (1998).

\bibitem{frau00}
{\sc U. Frauenfelder},
{\em Gromov convergence of pseudoholomorphic discs},
Diplomarbeit, ETH Z\"urich (2000).

\bibitem{geig08}
{\sc H. Geiges},
{\it An Introduction to Contact Topology},
Cambridge Stud. Adv. Math. {\bf 109},
Cambridge University Press, Cambridge (2008).

\bibitem{grom85}
{\sc M. Gromov},
Pseudoholomorphic curves in symplectic manifolds,
{\it Invent. Math.}
{\bf 82} (1985), 307--347.

\bibitem{hirs76}
{\sc M. W. Hirsch},
{\it Differential Topology},
Grad. Texts in Math. {\bf 33},
Springer-Verlag, Berlin (1976).

\bibitem{hofe93}
{\sc H. Hofer},
Pseudoholomorphic curves in symplectizations with applications to the
Weinstein conjecture in dimension three,
{\it Invent. Math.}
{\bf 114} (1993), 515--563.

\bibitem{hofe99}
{\sc H. Hofer},
Holomorphic curves and dynamics in dimension three,
in: {\it Symplectic Geometry and Topology}
(Park City, UT, 1997),
IAS/Park City Math. Ser. {\bf 7},
American Mathematical Society,
Providence, RI (1999), 35--101.

\bibitem{hls97}
{\sc H. Hofer, V. Lizan and J.-C. Sikorav},
On genericity for holomorphic curves in four-dimensional
almost-complex manifolds,
{\it J. Geom. Anal.}
{\bf 7} (1997), 149--159.

\bibitem{hoze94}
{\sc H. Hofer and E. Zehnder},
{\it Symplectic Invariants and Hamiltonian Dynamics},
Birkh\"auser Adv. Texts: Basler Lehrb\"ucher,
Birkh\"auser Verlag, Basel (1994).

\bibitem{kemi63}
{\sc M. A. Kervaire and J. W. Milnor},
Groups of homotopy spheres~I,
{\it Ann. of Math.~(2)\/}
{\bf 77} (1963), 504--537.

\bibitem{kosi93}
{\sc A.A. Kosinski},
{\em Differential Manifolds},
Academic Press, Boston, MA (1993).

\bibitem{lang99}
{\sc S. Lang},
{\it Fundamentals of Differential Geometry},
Grad. Texts in Math. {\bf 191},
Springer-Verlag, Berlin (1999).

\bibitem{lazz00}
{\sc L. Lazzarini},
Existence of a somewhere injective pseudo-holomorphic disc,
{\it Geom. Funct. Anal.}
{\bf 10} (2000), 829--862.

\bibitem{mcdu94}
{\sc D. McDuff},
Singularities and positivity of intersections of $J$-holomorphic curves,
in: {\it Holomorphic Curves in Symplectic Geometry},
Progr. Math. {\bf 117},
Birkh\"auser Verlag, Basel (1994), 191--215.

\bibitem{mcsa04}
{\sc D. McDuff and D. Salamon},
{\it $J$-holomorphic Curves and Symplectic Topology},
Amer. Math. Soc. Colloq. Publ. {\bf 52},
American Mathematical Society, Providence, RI (2004).

\bibitem{miwh95}
{\sc M. J. Micallef and B. White},
The structure of branch points in minimal surfaces and
in pseudoholomorphic curves,
{\it Ann. of Math.~(2)\/}
{\bf 141} (1995), 35--85.

\bibitem{munk60}
{\sc J. Munkres},
Differentiable isotopies on the $2$-sphere,
{\it Michigan Math. J.}
{\bf 7} (1960), 193--197.

\bibitem{munk64}
{\sc J. Munkres},
Obstructions to imposing differentiable structures,
{\it Illinois J. Math.}
{\bf 8} (1964), 361--376.

\bibitem{prwe84}
{\sc M. H. Protter and H. F. Weinberger},
{\it Maximum Principles in Differential Equations},
Springer-Verlag, Berlin (1967).

\bibitem{siko94}
{\sc J.-C. Sikorav},
Some properties of holomorphic curves in almost complex manifolds,
in: {\it Holomorphic Curves in Symplectic Geometry},
Progr. Math. {\bf 117},
Birkh\"auser Verlag , Basel (1994), 165--189.

\bibitem{smal59}
{\sc S. Smale},
Diffeomorphisms of the $2$-sphere,
{\it Proc. Amer. Math. Soc.}
{\bf 10} (1959), 621--626.

\bibitem{smal61}
{\sc S. Smale},
Generalized Poincar\'e's conjecture in dimensions greater than four,
{\it Ann. of Math.~(2)\/}
{\bf 74} (1961), 391--406.

\bibitem{wend05}
{\sc C. Wendl},
{\it Finite Energy Foliations and Surgery on Transverse Links},
Ph.D. Thesis, New York University
(2005).

\bibitem{ye98}
{\sc R. Ye},
Filling by holomorphic curves in symplectic $4$-manifolds,
{\it Trans. Amer. Math. Soc.}
{\bf 350} (1998), 213--250.

\bibitem{zehm03}
{\sc K. Zehmisch},
{\it The Eliashberg--Gromov Tightness Theorem},
Diplomarbeit, Universit\"at Leipzig (2003).

\end{thebibliography}
\end{document}